\documentclass[letterpaper,12pt]{amsart}
\usepackage{latexsym,array,delarray,amsmath,amsthm,amssymb,epsfig,setspace,tikz}
\usepackage{cite}
\usepackage{float}
\usepackage{nicematrix}

\usepackage{color}

\theoremstyle{plain}
\newtheorem{thm}{Theorem}[section]
\newtheorem{lemma}[thm]{Lemma}
\newtheorem{prop}[thm]{Proposition}
\newtheorem{cor}[thm]{Corollary}
\newtheorem{conj}[thm]{Conjecture}
\newtheorem*{thm*}{Theorem}
\newtheorem*{lemma*}{Lemma}
\newtheorem*{prop*}{Proposition}
\newtheorem*{cor*}{Corollary}
\newtheorem*{conj*}{Conjecture}

\theoremstyle{definition}
\newtheorem{defn}[thm]{Definition}
\newtheorem{ex}[thm]{Example}

\newtheorem{alg}[thm]{Algorithm}

\theoremstyle{remark}
\newtheorem*{rmk}{Remark}

\newcommand{\qq}{\mathbb{Q}}
\newcommand{\rr}{\mathbb{R}}

\newcommand{\kk}{\mathbb{K}}

\newcommand{\calc}{\mathcal{C}}

\newcommand{\cm}{\mathcal{M}}

\newcommand{\calp}{\mathcal{P}}

\newcommand{\ind}{\mbox{$\perp \kern-5.5pt \perp$}}

\newcommand{\dist}{\rm{dist}}
\newcommand{\inw}{\rm{in}_\omega}
\newcommand{\im}{\rm{im}}

\newcommand{\wcm}{\widetilde{\mathcal{M}}}

\title{Identifiable paths and cycles in linear compartmental models}
\author{Cashous Bortner}
\author{Nicolette Meshkat}

\begin{document}
\maketitle


\begin{abstract}
We introduce a class of linear compartmental models called \textit{identifiable path/cycle models} which have the property that all of the monomial functions of parameters associated to the directed cycles and paths from input compartments to output compartments are identifiable and give sufficient conditions to obtain an identifiable path/cycle model.  Removing leaks, we then show how one can obtain a locally identifiable model from an identifiable path/cycle model. These identifiable path/cycle models yield the \textit{only} identifiable models with certain conditions on their graph structure and thus we provide necessary and sufficient conditions for identifiable models with certain graph properties.  A sufficient condition based on the graph structure of the model is also provided so that one can test if a model is an identifiable path/cycle model by examining the graph itself.  We also provide some necessary conditions for identifiability based on graph structure. Our proofs use algebraic and combinatorial techniques.  


\end{abstract}

\maketitle

\section{Introduction}

The parameter identifiability problem is the question of whether or not the unknown parameters of a mathematical model can be determined from known data.  This paper is concerned with \textit{structural identifiability} 
analysis, that is, whether the model parameters can be identified from perfect input-output data (noise-free and of any duration required).
Structural identifiability is a necessary condition for \textit{practical identifiability}
which is identifiability analysis in the presence of noisy
and imperfect data.  Thus, structural identifiability is an important step in the parameter estimation problem, since failure to recover parameters in the ideal case implies failure in the imperfect case as well. If all of the parameters of a model can be determined, we say the model is (at least) locally identifiable, but if some subset of the parameters can take on an infinite number of values yet yield the same input-output data, the model is said to be unidentifiable.  

In this work, we examine a special class of models called \textit{linear compartmental models}.  Linear compartmental models are an important class of biological models used in the areas of cell biology, pharmacology, toxicology, ecology, physiology, and many other areas \cite{distefano-book}.  In a typical biological application, the mass or concentration of a substance (e.g. drug concentration in an organ) is represented by a compartment, and the transfer of material from one compartment to another is given by a constant rate parameter, called an \textit{exchange rate}.  The transfer of material from a compartment leaving the system is given by a constant rate parameter called the \textit{leak rate}, and any compartment containing such a leak is called a \textit{leak compartment}.  An input represents the input of material to a particular compartment of the system (e.g. IV drug input) and an output represents a measurement from a compartment (e.g. drug concentration in an organ), where such compartments are called \textit{input compartments} and \textit{output compartments}, respectively. The resulting ODE system of equations (see Equation (\ref{eq:main})) is linear.  This linearity feature has a nice mathematical consequence in that the model can be represented by a directed graph.  Thus, we can analyze identifiability problems in terms of the combinatorial 
structure of that graph.  

Recent work on identifiable reparametrizations \cite{MeshkatSullivant,BaaijensDraisma}, sufficient conditions for identifiability \cite{MeshkatSullivantEisenberg,submodel,Gerberding-Obatake-Shiu}, and identifiable functions of parameters \cite{MeshkatSullivant, meshkat-rosen-sullivant} has examined the combinatorial structure of the graph in a linear compartmental model to answer questions about what to do with an unidentifiable model.  One approach to dealing with an unidentifiable model is to reparametrize the model over identifiable functions of parameters in the model.  In other words, although not all the parameters are identifiable, one can attempt to reparametrize over a set of functions of parameters that can be determined from input-output data.  In \cite{MeshkatSullivant}, necessary and sufficient conditions were given to obtain an identifiable scaling reparametrization in the case of a linear compartmental model with a single input and output in the same compartment, leaks from every compartment, and having a strongly connected graph.  In \cite{MeshkatSullivantEisenberg}, these models were called \textit{identifiable cycle models} because the monomial functions associated to the directed cycles are identifiable and it was shown that removing all but one leak from such models results in identifiability.  Additionally, it was shown that removing a subset of leaks, but adding inputs or outputs to the remaining leak compartments, results in identifiability.  

In this paper, we expand upon the results in \cite{MeshkatSullivant, MeshkatSullivantEisenberg} in the following ways.  First, we consider the case of inputs and outputs not necessarily in the same compartment and define the analogous \textit{identifiable path/cycle model} (Definition \ref{defn:pathcycle}), which is a model where all the monomial functions associated to the directed cycles and paths from input to output are identifiable.  Just as in \cite{MeshkatSullivant}, this occurs when the model has a coefficient map whose image has maximal dimension (Theorem \ref{thm:idpathcycle}).  
We then take these identifiable path/cycle models and remove leaks from all compartments except input/output compartments to achieve identifiable models (Theorem \ref{thm:removeleaks}).  A similar result was demonstrated in \cite{MeshkatSullivantEisenberg}, but in that version, the intersection of input and output compartments was nonempty, whereas in the present work the input and output compartments need not coincide.  We then show that these identifiable path/cycle models yield the \textit{only} identifiable models with certain conditions on their graph structure (Theorem \ref{thm:addleaks}).  We thus provide necessary and sufficient conditions for identifiable models with certain graph properties (Corollary \ref{cor:necandsuff}).  We also give a sufficient condition for a model to be an identifiable path/cycle model which can be tested simply by examining the graph itself (Theorem \ref{thm:almost_isc}).  In addition, we weaken the conditions on the graph structure to obtain some necessary and sufficient conditions for identifiability (Corollary \ref{cor:necandsuffoutputconnectable}).  We also give some necessary conditions for identifiability in terms of the structure of the graph (Theorems \ref{thm:leakcondition}, \ref{thm:edgecondition}, \ref{thm:pathcondition}).  Finally, we give a construction of identifiable models using results from \cite{BaaijensDraisma} (Algorithm \ref{alg:construct}).

Our results apply to a large class of linear compartmental models which arise in many real-world applications. Path models of the form in Proposition \ref{prop:path} arise in physiological models involving metabolism, biliary, or excretory pathways \cite{distefano-book} and models of neuronal dendritic trees \cite{bressloff1993compartmental}.  Path models also arise when modeling the delayed response to input and are called \textit{time-delay models} \cite{distefano-book}.  One such example is Example 4.13 from \cite{distefano-book} on oral dosing losses and delays in the gastrointestinal tract.  Some other path models are considered in Section \ref{section:Examples}.  More generally, we consider models that are \textit{strongly input-output connected}.  Mammillary and catenary models \cite{distefano-book} fall into this category, as well as a variation of mammillary and catenary models where input and output are in distinct neighboring compartments but the edge from output to input is missing (see Figure \ref{fig:ex}). More generally, our results apply to models that can be thought of as path models combined with catenary models and are considered in Section \ref{section:Examples}.  Such a model could, for example, represent a time-delay model coupled with a catenary model.

The organization of the paper is as follows. Section \ref{section:background} gives the necessary background.  Section \ref{section:idpathcycle} gives the definition of an identifiable path/cycle model and how to obtain one.  Section \ref{section:classification} gives a classification of all identifiable models with certain graph properties.  Section \ref{section:othered} examines weaker conditions on the graph structure for necessary and sufficient conditions for identifiability.  Section \ref{section:neccond} gives necessary conditions for identifiability in terms of the graph structure of the model.  Section \ref{section:Examples} demonstrates our results on some real-world examples.  Section \ref{section:computations} gives computations on the number of models with maximal dimension with a certain number of inputs and outputs.  Section \ref{section:construction} gives a construction of identifiable models.  Section \ref{section:conclusion} gives a conclusion and a conjecture on identifiable scaling reparametrizations of identifiable path/cycle models.

\begin{table}[h]
\centering
\begin{tabular}{| c | c |}
\hline
{\bf Result }   & {\bf Explanation} \\
\hline
Corollary \ref{cor:necandsuff}   &  Gives necessary and sufficient conditions for a \\
& strongly input-output connected model to be \\
& an identifiable path/cycle model      
      \\ \hline
Theorem \ref{thm:almost_isc}    &      Gives a sufficient condition to be an   \\ 
& identifiable path/cycle model based on graph structure   \\  \hline
Corollary \ref{cor:necandsuffoutputconnectable}      &     Gives necessary and sufficient conditions for an  \\
& output connectable model to be \\
& an identifiable path/cycle model    \\      \hline
\end{tabular}
\caption{Summary of main results.}
\label{tab:summary}
\end{table}

\begin{table}[h]
\centering
\begin{tabular}{| c | c | c |}
\hline
{\bf Prior Result }   & {\bf New Result} & {\bf Explanation} \\
\hline
Theorem 1.2 of \cite{MeshkatSullivant}   &       	
Theorem \ref{thm:idpathcycle}       &       Generalizes conditions for  \\
& & identifiable cycle model to \\
&& identifiable path/cycle model      
      \\ \hline
Theorem 5 from \cite{MeshkatSullivantEisenberg}         &     	
Theorem \ref{thm:removeleaks}                       &       Generalizes removing leaks \\
&&to obtain identifiability    \\      \hline
Theorem 5.13 from \cite{MeshkatSullivant}         &     	
Theorem \ref{thm:almost_isc}                       &       Generalizes inductively strongly  \\
& & connected to almost inductively \\
&& strongly connected    \\      \hline
Proposition 5.4 from \cite{MeshkatSullivant}         &     	
Proposition \ref{prop:path}                       &       Generalizes identifiable cycle \\
& & to identifiable path    \\      \hline
Proposition 5.5 from \cite{MeshkatSullivant}         &     	
Proposition \ref{prop:addvertex}                       &       Generalizes adding a new vertex \\     \hline
Proposition 5.3 from \cite{MeshkatSullivant}         &     	
Theorem \ref{thm:edgecondition}                       &       Generalizes necessary condition of having  \\
& Theorem \ref{thm:pathcondition} &an exchange to having a path   \\      \hline
\end{tabular}
\caption{Summary of which new results in this paper generalize the prior results from \cite{MeshkatSullivant} and \cite{MeshkatSullivantEisenberg}.}
\label{tab:comparison}
\end{table}

\section{Background} \label{section:background}  

Let $G$ be a directed graph with vertex set $V$ and set of 
directed edges $E$.  Each vertex $i \in V$ corresponds to a
compartment in our model and an edge $j \rightarrow i$ denotes 
a direct flow of material from compartment $j$ to
compartment $i$.  Also introduce three subsets of the vertices
$In, Out, Leak \subseteq V$ corresponding to the
set of input compartments, output compartments, and leak compartments
respectively.  To each edge $j \rightarrow i$ we associate
an independent parameter $a_{ij}$, the rate of flow
from compartment $j$ to compartment $i$.  
To each leak node $i \in Leak$, we associate an independent
parameter $a_{0i}$, the rate of flow from compartment $i$ leaving the system.

We associate a matrix $A(G)$, called the \textit{compartmental matrix} to the graph and the set $Leak$
 in the following way:
\[
  A(G)_{ij} = \left\{ 
  \begin{array}{l l l}
    -a_{0i}-\sum_{k: i \rightarrow k \in E}{a_{ki}} & \quad \text{if $i=j$ and } i \in Leak\\
        -\sum_{k: i \rightarrow k \in E}{a_{ki}} & \quad \text{if $i=j$ and } i \notin Leak\\
    a_{ij} & \quad \text{if $j\rightarrow{i}$ is an edge of $G$}\\
    0 & \quad \text{otherwise}\\
  \end{array} \right.
\]
For brevity, we will often
 use $A$ to denote $A(G)$.  Also, define the vector $\mathcal{A} \in \rr^{|E|+|Leak|}$ consisting of nonzero parameters of $A$.

Then we construct a system of linear ODEs with inputs and outputs associated to the quadruple
$(G, In, Out, Leak)$ as follows:
\begin{equation} \label{eq:main}
x'(t)=Ax(t)+u(t)  \quad \quad y_i(t)=x_i(t)  \mbox{ for } i \in Out
\end{equation}
 where $u_{i}(t) \equiv 0$ for $i \notin In$.
 The coordinate functions $x_{i}(t)$ are the state variables, the 
 functions $y_{i}(t)$ are the output variables, and the nonzero functions $u_{i}(t)$ are
 the inputs.  The resulting model is called a   \textit{linear compartmental model}. 

We will indicate output compartments by this symbol: \begin{tikzpicture}[scale=0.7]
 	\draw (4.66,-.49) circle (0.05);	
	 \draw[-] (5, -.15) -- (4.7, -.45);	
\end{tikzpicture} .  
Input compartments are labeled by ``in'', and leaks are indicated by edges which go to no vertex.

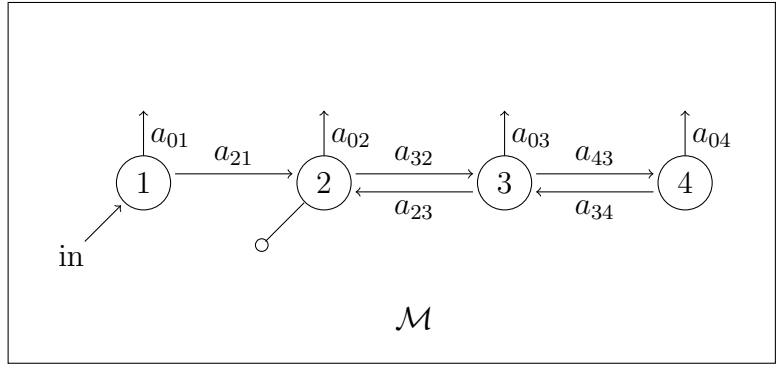
\begin{figure}[H]
\begin{center}
	\begin{tikzpicture}[scale=1.2]
 	\draw (0,0) circle (0.3);
 	\draw (2,0) circle (0.3);
 	\draw (4,0) circle (0.3);
 	\draw (6,0) circle (0.3);
    	\node[] at (0, 0) {1};
    	\node[] at (2, 0) {2};
    	\node[] at (4, 0) {$3$};
    	\node[] at (6, 0) {$4$};
	 \draw[->] (0.35, .1) -- (1.65, .1);
	 \draw[->] (2.35,.1) -- (3.65,.1);
	 \draw[<-] (2.35,-.1) -- (3.65,-.1);
	 \draw[->] (4.35,.1) -- (5.65,.1);
	 \draw[<-] (4.35,-.1) -- (5.65,-.1);
   	 \node[] at (1, 0.3) {$a_{21}$};
	\node[] at (3,0.3) {$a_{32}$};
	\node[] at (3,-.3) {$a_{23}$};
	\node[] at (5,.3) {$a_{43}$};
	\node[] at (5,-.3) {$a_{34}$};
	\draw (1.31,-.69) circle (0.07);	
	 \draw[-] (1.35, -.65 ) -- (1.78, -.22);	
	 \draw[->] (-.65, -.65) -- (-.25, -.25);	
   	 \node[] at (-.8,-.8) {in};
	 \draw[->] (0,.3) -- (0, .8);	
   	 \node[] at (.3, .5) {$a_{01}$};
	 \draw[->] (2,.3) -- (2, .8);	
   	 \node[] at (2.3, .5) {$a_{02}$};
	 \draw[->] (4,.3) -- (4, .8);	
   	 \node[] at (4.3, .5) {$a_{03}$};
	 \draw[->] (6,.3) -- (6, .8);	
   	 \node[] at (6.3, .5) {$a_{04}$};
\draw (-1.5,-2) rectangle (7, 2);
    	\node[] at (3, -1.5) {$\cm$};
    	

\end{tikzpicture}
\end{center}
\caption{Graph for Example \ref{ex:continuing}.}
\label{fig:ex}
\end{figure}

\begin{ex} \label{ex:continuing} The model $\cm=(G,\{1\},\{2\},V)$ with $G$ given 
in Figure \ref{fig:ex} is a linear compartmental model with equations given by:
{\small
\begin{align} \label{eq:A-ex}
\begin{pmatrix}
x_1' \\
x_2' \\
x_3' \\
x_4' 
\end{pmatrix} 
&~=~
\begin{pmatrix}
- a_{01} - a_{21} & 0 & 0 & 0 \\
a_{21} & -a_{02}-a_{32} & a_{23} & 0\\
0 & a_{32} & -a_{03}-a_{23}-a_{43} & a_{34} \\
0 & 0 & a_{43} & -a_{04} -a_{34}
\end{pmatrix}
\begin{pmatrix}
x_1 \\
x_2 \\
x_3 \\
x_4
\end{pmatrix} +
\begin{pmatrix}
u_1 \\
0 \\
0 \\
0
\end{pmatrix}~,
\end{align}
}

with 
output equation $y_2=x_2$. 
\end{ex}

For a model $(G, In, Out, Leak)$ where there is a leak in every
compartment (i.e. $Leak = V$), it can greatly simplify the representation
to use the fact that the diagonal entries of $A(G)$ are the
only places where the parameters $a_{0i}$ appear.
Since these are algebraically independent parameters,
we can introduce a new algebraically  independent 
parameter $a_{ii}$ for the diagonal entries (i.e. we make the substitution $a_{ii}  =  -a_{0i}-\sum_{k: i \rightarrow k \in E}{a_{ki}}$)
to get generic parameter values along the diagonal.
Identifiability questions in such a model are equivalent
to identifiability questions in the model with this reparametrized
matrix. 

We will be considering graphs that have some special connectedness properties.  We define these properties now, as well as the basic algebraic structures (monomial paths and cycles) we will be working over. 

\begin{defn} A directed graph $G$ is \textit{connected} if each pair of vertices in the graph is joined by an undirected path.  A directed graph $G$ is \textit{strongly connected} if there exists a directed path from each vertex to every other vertex.  A directed graph $G$ is \textit{inductively strongly connected} with respect to vertex $1$ if each of the induced subgraphs $G_{\{1, \ldots, i\}}$ is strongly connected for $i = 1, \ldots, n$ for some ordering of the vertices $1,\ldots,i$ which must start at vertex $1$.
\end{defn}

\begin{defn} A \textit{closed path} in a directed graph $G$ is a sequence of 
vertices $i_{0},i_{1}, i_{2}, \ldots, i_{k}$ with $i_{k} = i_{0}$ and
such that $i_{j+1} \to i_{j}$ is an edge for all $j = 0, \ldots, k-1$.
A \textit{cycle} in $G$ is a closed path with no repeated vertices.
To a cycle $C = i_{0},i_{1}, i_{2}, \ldots, i_{k}$, we associate the
monomial $a^{C} = a_{i_{0}i_{1}}a_{i_{1}i_{2}}\cdots a_{i_{k}i_{0}}$,
which we refer to as a \textit{monomial cycle}.  If a monomial cycle $a^{C}$ has
length $k$, we refer to it as a $k$-cycle.  
\end{defn}

Note that we also include the monomial cycles $a_{ii}$ which
are $1$-cycles, or \textit{self-cycles}. 

\begin{defn} A \textit{path} from vertex $i_{k}$ to vertex $i_{0}$ in a directed graph $G$ is a sequence of 
vertices $i_{0},i_{1}, i_{2}, \ldots, i_{k}$
such that $i_{j+1} \to i_{j}$ is an edge for all $j = 0, \ldots, k-1$.
To a path $P = i_{0},i_{1}, i_{2}, \ldots, i_{k}$, we associate the
monomial $a^{P} = a_{i_{0}i_{1}}a_{i_{1}i_{2}}\cdots a_{i_{k-1}i_{k}}$,
which we refer to as a \textit{monomial path}.  If a monomial path $a^{P}$ has
length $k$, we refer to it as a $k$-path.  
\end{defn}

\begin{rmk} We will sometimes drop the word ``monomial'' from ``monomial path'' or ``monomial cycle'' and simply refer to these as paths and cycles, as is the case in the next definition.
\end{rmk}

We now define the path/cycle map for a linear compartmental model $\cm=(G,In,Out,V)$:

\begin{defn} Let $\calp=\calp(G)$ be the set of all directed cycles and paths from input to output vertices in the graph $G$.  Define the \textit{path/cycle} map by:

\begin{equation} \label{eq:pi}
\pi: \rr^{|E|+|V|} \to \rr^{|\calp|} , A \mapsto (a^C)_{C \in \calp}
\end{equation}
\end{defn}

Now we give some definitions from \cite{submodel} regarding an important subgraph to this work:

\begin{defn} For a linear compartmental model $\cm=(G,In,Out,Leak)$, let $i \in Out$.  The \textit{output-reachable subgraph to $i$} (or \textit{to $y_i$}) is the induced subgraph of $G$ containing all vertices $j$ for which there is a directed path in $G$ from $j$ to $i$. A linear compartmental model is \textit{output connectable} if every compartment has a directed path leading from it to an output compartment. 
\end{defn}

We add the following definition:

\begin{defn} A linear compartmental model $\cm=(G,In,Out,Leak)$ is \textit{output connectable to every output} if every compartment has a directed path leading from it to every output compartment.  
\end{defn}

We will be using the so-called \textit{differential algebra approach} to structural identifiability \cite{Ljung,Ollivier}.  Other methods to test identifiability include the Taylor Series approach \cite{Pohjanpalo1978}, the generating series approach \cite{Walter1982}, the similarity transformation approach \cite{Vajda1989}, and the observability-identifiability condition \cite{Villaverde}.  In the differential algebra approach, we view the model equations as differential polynomials in a differential polynomial ring $R(p)[u,y,x]$, i.e., the ring of polynomials in state variable vector $x$, output vector $y$, input vector $u$, and their derivatives, with coefficients in $R(p)$ for parameter vector $p$. Since the unmeasured state variables $x_i$ cannot be determined, the goal in this approach is to use differential elimination to eliminate all unknown state variables and their derivatives.  The resulting equations are only in terms of input variables, output
variables, their derivatives, and parameters, so these equations have the following form:
\begin{align} \label{eq:general-i-o}
\sum_i{c_i(p)\Psi_i(u,y)} =0~.
\end{align}
 
An equation of the form~\eqref{eq:general-i-o} is called an \textit{input-output equation} for
$\mathcal M$. 

For nonlinear models, one standard ``reduced'' generating set for these input-output equations is formed by those equations in a \textit{characteristic set} (defined precisely in~\cite{glad})
that do not involve the $x_i$'s or their derivatives.
In a characteristic set, which can be computed using the software {\tt DAISY}~\cite{daisy},
each $\Psi_i(u,y)$ in each input-output equation~\eqref{eq:general-i-o} 
is a differential monomial, i.e., a monomial purely
in terms of input variables, output variables, and their derivatives. The terms $c_i(p)$ are called the coefficients of the input-output equations.  These coefficients can be fixed uniquely by normalizing the input-output equations to make them monic \cite{daisy}.

However, for linear models, it has been shown that these input-output equations can be found much more easily using the Transfer Function approach \cite{Bellman} or even a trick with Cramer's Rule \cite{MeshkatSullivant}.  We will be taking the latter approach to get an explicit formula for the input-output equations.  

We now state Theorem 3.8 from \cite{submodel} with input $i$ and output $j$, which gives the input-output equation in $y_j$ in terms of the output-reachable subgraph to $y_j$.

\begin{thm} \label{thm:ioscc} 
Let $\mathcal{M}=(G, In, Out, Leak)$
be a linear compartmental model with at least one input.
Let $j \in Out$, 
and assume that there exists a directed path from some input compartment 
to compartment-$j$. 
Let $ H$ denote the output-reachable subgraph to $y_j$, let $A_H$ denote the compartmental matrix for the restriction $\mathcal{M}_H$, and let $\partial I$ be the the product of the differential operator $d/dt$ and the $|V_G| \times |V_G|$ identity matrix.  
Then the following is an input-output equation for $\mathcal M$ involving $y_j$:
 \begin{align}  \label{eq:i-o-for-general-model}
 	\det (\partial I -{A}_H) y_j ~=~   \sum_{i \in In \cap V_H} (-1)^{i+j} \det \left( \partial I-{A}_H \right)_{ij} u_i ~,
	 \end{align}
where $ \left( \partial I-{A}_H \right)_{ij}$ denotes the matrix obtained from
 $\left( \partial I-{A}_H \right)$ by removing the row corresponding to compartment-$i$ and the column corresponding to compartment-$j$.
Thus, this input-output equation~\eqref{eq:i-o-for-general-model}
involves only the output-reachable subgraph to $y_j$.
\end{thm}

\begin{ex} [Continuation of Example \ref{ex:continuing}] The model $\cm=(G,\{1\},\{2\},V)$ with $G$ given by the graph \{ $1 \rightarrow 2, 2 \rightarrow 3, 3 \rightarrow 2, 3 \rightarrow 4, 4 \rightarrow 3$ \} has leaks from every compartment, thus writing the diagonal elements as $a_{ii}$, we have the following input-output equation: 

{\footnotesize
\begin{align*} 
	& y_2^{(4)} +(-a_{11}-a_{22}-a_{33}-a_{44}) y_2^{(3)}
	\\ \notag
	& \quad 
	+(a_{11} a_{22} - a_{23} a_{32} + a_{11} a_{33} + a_{22} a_{33} - a_{34} a_{43} + 
 a_{11} a_{44}  + a_{22} a_{44} + a_{33} a_{44}) y_2''
	\\ \notag
	 &\quad 
	+(a_{11} a_{23} a_{32} - a_{11} a_{22} a_{33} + a_{11} a_{34} a_{43} + a_{22} a_{34} a_{43} - a_{11} a_{22} a_{44} + a_{23} a_{32} a_{44} - a_{11} a_{33} a_{44}- a_{22} a_{33} a_{44} )y_2'
	\\ \notag
	 & \quad 
	+  (-a_{11} a_{22} a_{34} a_{43} - a_{11} a_{23} a_{32} a_{44} + a_{11} a_{22} a_{33} a_{44})y_2 
		\\ \notag
	  &
	=~ (a_{21})u_1''
	\\ \notag
	 & \quad 
	+(-a_{21}a_{33}-a_{21}a_{44})u_1'
		\\ \notag
	 & \quad 
	+(a_{21}a_{33}a_{44}-a_{21}a_{34}a_{43})u_1~.
\end{align*}
}

\end{ex}

\subsection{Identifiability}

In the differential algebra approach to structural identifiability, one tests identifiability by using the input-output equations determined from the \textit{characteristic set} \cite{daisy}.  However, for the case of a linear model, we have shown that the input-output equations can be formed using Equation (\ref{eq:i-o-for-general-model}) and thus we would like to define identifiability using these equations.   An important question arises here: can we use the input-output equations in Equation (\ref{eq:i-o-for-general-model}) to test identifiability?  

In the case of a single output, there is just a single input-output equation of the form in Equation (\ref{eq:i-o-for-general-model}).  However, in the case of multiple outputs, there is an input-output equation for each output in $Out$.  The equations formed in Theorem \ref{thm:ioscc} are not necessarily minimal, i.e. of lowest degree, but this condition of minimality is required in order for identifiability to be well-defined \cite{saccomani2003parameter}.  In Theorem 3 of \cite{OvchinnikovPogudinThompson}, it is shown that the input-output equations in Equation (\ref{eq:i-o-for-general-model}) can be used to analyze identifiability in the case of a linear compartmental model with at least one input and whose graph is strongly connected.  Thus, in the rest of this paper, we will be assuming $G$ is strongly connected and the model has at least one input when we are analyzing models with more than 1 output so that our definition of identifiability is well-defined.


The next step of the differential algebra approach assumes that the coefficients $c_i(p)$ of the input-output equations
can be recovered uniquely from input-output data, and thus are 
presumed to be known quantities~\cite{SoderstromStoica}.  While there are examples of linear models where this assumption can lead to incorrect conclusions about identifiability (see Example 2.14 of \cite{global-id}), this assumption that the coefficients can be recovered uniquely from input-output data holds for all linear compartmental models with corresponding graphs such that one can reach a leak or an input from every vertex (see Theorem 2 of \cite{OvchinnikovPogudinThompson}). We will only be considering either strongly connected or strongly input-output connected models where the output compartment contains a leak, so the coefficients of the input-output equations can be uniquely recovered.

We give some definitions of identifiability from \cite{MeshkatSullivantEisenberg}.

\begin{defn}\label{defn:identify}
Let $(G, In, Out, Leak)$ be a linear compartment model and
let $c$ denote the vector of all nonzero and nonmonic coefficient functions of
all the linear input-output equations derived in Theorem
\ref{thm:ioscc} for each $i \in Out$.  The function $c$ defines a map
$c:  \rr^{|E| + |Leak|}  \rightarrow \rr^{k}$,
where $k$ is the total number of coefficients which we call the \textit{coefficient map}.  The linear compartment model  $(G, In, Out, Leak)$ is: 
\begin{itemize}
	\item{\textit{globally identifiable} if $c$ is a one-to-one function, and is \textit{generically globally identifiable} if global identifiability holds everywhere in $\rr^{|E| + |Leak|}$, except possibly on a set of measure zero.}
	\item{\textit{locally identifiable} if around any neighborhood of a point in $\rr^{|E| + |Leak|}$, $c$ is a one-to-one function, and is \textit{generically locally identifiable} if local identifiability holds everywhere in $\rr^{|E| + |Leak|}$, except possibly on a set of measure zero.}
\item{\textit{unidentifiable} if $c$ is infinite-to-one.}
\end{itemize}
\end{defn}

\begin{rmk} Throughout the rest of this paper, we will be concerned with generic local identifiability.  Thus, we will drop the word ``generic'' and just state a model is locally identifiable.  Additionally, if a model is ``identifiable'', this means it is (at least) locally identifiable.
\end{rmk}

Since we will be concerned with generic local identifiability in this work, we will be using the following proposition from \cite{MeshkatSullivantEisenberg} which follows from the fact that the rank of the Jacobian of $c$ evaluated at a generic point
is equal to the dimension of the image of $c$ \cite[Prop. 14.4]{Harris}.

\begin{prop} [Proposition 2 from \cite{MeshkatSullivantEisenberg}] \label{prop:mse}
The model $(G, In, Out, Leak)$ is generically locally
identifiable if and only if the rank of the Jacobian of $c$
is equal to $|E| + |Leak|$ when evaluated at a random point. 
\end{prop}

If a model is unidentifiable, then this means that not all of the parameters can be determined (uniquely or finitely).  However, we may still be interested in finding \textit{identifiable functions of parameters} or \textit{identifiable combinations}, and in using these functions to attempt to reparametrize the model.  

\begin{defn} \label{defn:idfunction}
Let $c$ be a function $c:\rr^{|E| + |Leak|}  \rightarrow \rr^{k}$.  A function
$f : \rr^{|E|+|Leak|} \rightarrow \rr$ is \textit{globally identifiable} from
$c$ if there exists a function $\Phi: \rr^{k} \rightarrow \rr$
such that $\Phi \circ c  = f$.  The function $f$ is
\textit{locally identifiable} if there is a finitely multivalued function
$\Phi: \rr^{k} \rightarrow \rr$
such that $\Phi \circ c  = f$.
\end{defn}

\subsection{Strongly input-output connected}

In order to consider identifiable path/cycle models, we will be considering graphs $G$ that have the special property of being connected and every edge is contained in a cycle or path from input to output.  We call this \textit{strongly input-output connected}:

\begin{defn} We say a graph $G$ is \textit{strongly input-output connected} if it is connected and every edge is contained in a cycle or path from input to output.
\end{defn}

We first show that, in the case of a single output, being strongly input-output connected implies being output connectable, so if we assume the former we get output connectable and can use Theorem \ref{thm:ioscc} with the whole matrix $A$.  Likewise, for the case of multiple outputs, we show that being strongly connected implies being output connectable to every output.

\begin{prop} \label{prop:outputreach} (1) Consider a model $\cm=(G, In, \{j\}, Leak)$. Assume $G$ is strongly input-output connected.  Then $G$ is output connectable.  (2) Now consider a model $\cm=(G, In, Out, Leak)$.  Assume $G$ is strongly connected.  Then $G$ is output connectable to every output.
\end{prop}

\begin{proof} Let $\cm=(G, In, \{j\}, Leak)$.  Assume $G$ is strongly input-output connected, i.e. it is connected and every edge is contained in a cycle or path from input to output.  Since every edge contained in a path from input to output is connected to the output, we need only consider the edges in cycles.  If a vertex in a cycle coincides with a vertex on a path from input to output, then we are done.  Thus, assume that there exists a cycle whose vertices do not intersect with the vertices on paths from input to output.  Since the graph is connected, the cycle must be attached via a directed edge from either the cycle to a path from input to output or vice versa.  But the attaching edge must also be on a path from input to output. Thus for any edge from the path to the cycle, there must be a corresponding edge from the cycle to the path.  Thus the graph is output connectable.

If $\cm=(G, In, Out, Leak)$ and $G$ is strongly connected, then since there is a path from each vertex to every other vertex, then $G$ is output connectable to every output.
\end{proof}

\begin{rmk} Note that in Proposition \ref{prop:outputreach}, $G$ must be strongly connected as opposed to strongly input-output connected when $\cm=(G, \{i\}, Out, Leak)$, or else not every vertex may connect to \textit{every} output.
\end{rmk}

\begin{rmk} A model that is strongly input-output connected in the case of a single output or strongly connected in the case of multiple outputs is always \textit{structurally observable} \cite{godfrey-chapman}, as it is output connectable to every output by Proposition \ref{prop:outputreach}.  
\end{rmk}

We now show that the property of being strongly input-output connected is \textit{almost} strongly connected, in the sense that the graph becomes strongly connected once an edge is added (if not already there) from the output to every input if $\cm=(G, In, \{j\}, Leak)$ or an edge is added from every output to the input if $\cm=(G, \{i\}, Out, Leak)$.  

\begin{prop} (1) Consider a model $\cm=(G, In, \{j\}, Leak)$.  The model $\cm$ is strongly connected if an edge is added from output $j$ to every input if and only if it is strongly input-output connected. (2) Now consider a model $\cm=(G, \{i\}, Out, Leak)$.  The model $\cm$ is strongly connected if an edge is added from every output to input $i$ if and only if it is strongly input-output connected.  Strongly connected implies strongly input-output connected.
\end{prop}

\begin{proof} A model $\cm=(G, In, \{j\}, Leak)$ is strongly connected if and only if it is connected and every edge is contained in a cycle.  Thus a model $\cm=(G, In, \{j\}, Leak)$ is strongly connected if a path from output $j$ to every input is added if and only if it is connected and every edge is contained in a cycle or path from input to output, i.e. strongly input-output connected.  Likewise, a model $\cm=(G, \{i\}, Out, Leak)$ is strongly connected if a path from every output to input $i$ is added if and only if it is connected and every edge is contained in a cycle or path from input to output, i.e. strongly input-output connected.  Additionally, if a model is strongly connected, then it is strongly input-output connected, as every edge is contained in a cycle. 
\end{proof}

We can also examine the \textit{minimum} number of edges in order to be either strongly connected or strongly input-output connected:

\begin{prop} If $G$ is strongly connected, the minimum number of edges is $|V|$.  If $G$ is strongly input-output connected for input $i$ and output $j$, the minimum number of edges is $|V|-1$.
\end{prop}

\begin{proof} For $G$ to be strongly connected, each vertex must have at least one incoming and one outgoing edge.  Thus the minimum number of edges is $|V|$.  
 If a graph is strongly input-output connected for input $i$ and output $j$, then this means it becomes strongly connected if an edge is added from output $j$ to input $i$ (if not already there).  This means the minimum number of edges is $|V|-1$.
\end{proof}

\subsection{Expected number of coefficients}

We first give a result from \cite{MeshkatSullivantEisenberg}, which we have reworded to agree with the new terminology in this work and have split into two parts: Proposition \ref{prop:factor} shows that the coefficient map factors through, i.e. can be written purely in terms of, cycles, self-cycles and paths, and Lemma \ref{lemma:hot} gives the degree of the highest-order term on the right hand side of the input-output equation.

\begin{prop} [Proposition 5 from \cite{MeshkatSullivantEisenberg}] \label{prop:factor} Let $\cm=(G,In,\{j\},V)$ represent a linear compartmental model that is output connectable. The coefficient map $c$ factors through cycles, self-cycles, and paths from input to output.  
\end{prop}

\begin{proof} Let $\calc(G)$ be the set of all cycles in $G$, corresponding to a matrix $A$.  Recall that the coefficients of the characteristic polynomial of $A$ can be written as

$$c_{i} =  (-1)^i\sum_{C_{1}, \ldots, C_{k} \in \calc(G)} \prod_{j = 1}^{k}  {\rm sign}(C_{j}) a^{C_{j}},$$

where the sum is over all collections of vertex disjoint cycles
involving exactly $i$ edges of $G$, and ${\rm sign}(C) = 1$ if
$C$ is odd length and ${\rm sign}(C) = -1$ if $C$ is even length. This means for every $i$, all cycles of length $i$ appear as monomial terms in $c_{i}$, and for $j > i$, these cycles of length $i$ appear as monomial products with other cycles in $c_{j}$.

By Theorem \ref{thm:ioscc} and the fact that $G$ is output connectable, meaning that the output-reachable subgraph of $G$ is all of $G$, the input-output equation for $y_j$ is given by:
\begin{align} \label{eq:output-connectable-i-o}
\det(\partial{I}-A)y_j=\sum_{i \in In} (-1)^{i +j}\det(\partial{I}-A)_{ij} u_i.
\end{align}

This means the coefficients on the left hand side factor through the cycles in $G$.  

Let us now examine these coefficients of the $u_i$ terms in Equation (\ref{eq:output-connectable-i-o}).  For $i=j$, the term $\det(\partial{I}-A)_{ii}$ gives the coefficients of the characteristic polynomial for the matrix $A_{ii}$ with row $i$ and column $i$ removed, thus these coefficients factor through cycles of the induced subgraph removing vertex $i$. 

Now assume $i\neq{j}$.  The characteristic polynomial of $A$ can be determined by expanding $\det(\partial{I}-A)$ along the $i^{th}$ row.  Let $\tilde{A}$ be the matrix $A$ with the entry $a_{ij}$ nonzero.  Then for $i\neq{j}$, taking the partial derivative of the characteristic polynomial of $\tilde{A}$ with respect to $a_{ij}$ precisely gives the polynomial $\det(\partial{I}-A)_{ij}$, up to a minus sign.    Since the coefficients of the characteristic polynomial of $\tilde{A}$ factor through the cycles, then taking the derivative of these coefficients with respect to $a_{ij}$ has the effect of removing all monomial terms not involving $a_{ij}$ and setting $a_{ij}$ to one in the monomial terms that do involve $a_{ij}$.  This effectively transforms all cycles involving $a_{ij}$ to paths from the $i^{th}$ vertex to the $j^{th}$ vertex.  Thus, each of the monomial terms are products of paths from the $i^{th}$ vertex to the $j^{th}$ vertex, cycles, and self-cycles.  In other words, coefficients are of the form:

$$c_{m} =  (-1)^m\sum_{P_{1}, \ldots, P_{n} \in \calp(G)} \prod_{l = 1}^{n}  {\rm sign}(P_{l}) a^{P_{l}},$$
where the sum is over all collections of vertex disjoint cycles and paths from $i$ to $j$
involving exactly $m$ edges of $G$, and ${\rm sign}(P) = 1$ if
$P$ is odd length and ${\rm sign}(P) = -1$ if $P$ is even length. 

Thus the coefficients can be factored over cycles, self-cycles, and paths from input to output. In other words, there is a polynomial map 
\begin{equation}\label{eq:psi}
\psi: \rr^{|\calp|} \to \rr^k
\end{equation} where $k$ is the number of coefficients, such that $c=\psi \circ \pi$ where $\pi$ is the path/cycle map from Equation \ref{eq:pi}. 
\end{proof}

We will be writing the number of coefficients in terms of the minimal distance between an input and output compartment.  We define this now:

\begin{defn} \label{defn:dist} Let $i$ be an input compartment and let $j$ be an output compartment, $i\neq{j}$.  Let $\mathcal{P}(i,j)$ be the set of all paths from vertex $i$ to vertex $j$.  Let $l(P)$ denote the length of a path $P \in \mathcal{P}$.  Then we can define the minimum length of all paths from vertex $i$ to vertex $j$ as $\dist(i,j) = \min_{P \in \mathcal{P}(i,j)}$ $ l(P)$.
\end{defn}

\begin{lemma} [Proposition 5 from \cite{MeshkatSullivantEisenberg}]  \label{lemma:hot} Let $\cm=(G,In,\{j\},V)$ represent a linear compartmental model with $\cm$ output connectable. The highest-order term in $u_i$ where $i \in In$ on the right hand side of the input-output equation, Equation (\ref{eq:output-connectable-i-o}), is of degree $|V|-1-\dist(i,j)$. 
\end{lemma}

\begin{proof} 
Let $\tilde{A}$ be the matrix $A$ with the entry $a_{ij}$ nonzero. Let $\calc(\tilde{G})$ be the set of all cycles in $\tilde{G}$, corresponding to a matrix $\tilde{A}$.  To determine the coefficient of the highest-order term in $u_i$, recall that the coefficients of the characteristic polynomial of $\tilde{A}$ can be written as

$$c_{m} =  (-1)^m\sum_{C_{1}, \ldots, C_{k} \in \calc(\tilde{G})} \prod_{l = 1}^{k}  {\rm sign}(C_{l}) a^{C_{l}},$$
where the sum is over all collections of vertex disjoint cycles
involving exactly $i$ edges of $\tilde{G}$, and ${\rm sign}(C) = 1$ if
$C$ is odd length and ${\rm sign}(C) = -1$ if $C$ is even length. This means for every $m$, all cycles of length $m$ appear as monomial terms in $c_{m}$, and for $l > m$, these cycles of length $m$ appear as monomial products with other cycles in $c_{l}$.

We now determine the highest-order term in $u_i$.  Since $\det(\partial{I}-A)_{ij}$ is just the partial derivative of the characteristic polynomial of $\tilde{A}$ with respect to $a_{ij}$, up to a minus sign, then the right-hand side of the input-output equation for output $y_j$ is of the form, where $n=|V|$:
$$
\sum_{i \in In} (-1)^{i+j}\left( \frac{\partial{c_{1}}}{\partial{a_{ij}}}u_{i}^{(n-1)}+\frac{\partial{c_{2}}}{\partial{a_{ij}}}u_{i}^{(n-2)}+\frac{\partial{c_{3}}}{\partial{a_{ij}}}u_{i}^{(n-3)}+\cdots+\frac{\partial{c_{n}}}{\partial{a_{ij}}}u_{i} \right)
$$

We note that not all of these coefficients $\frac{\partial{c_{k}}}{\partial{a_{ij}}}$ for $k =1, ..., n$ are nonzero and thus we must determine the first nonzero coefficient.  

Recall Definition \ref{defn:dist} for the minimal distance between $i$ and $j$.  Let the length of the shortest cycle involving $a_{ij}$ be of length $\dist(i,j)+1$, so that the length of the shortest path from $i$ to $j$ is of length $\dist(i,j)$.  Then the coefficient of the highest-order term in $u_i$ is $\partial{c_{\dist(i,j)+1}}/\partial{a_{ij}}$, which is a sum of the shortest paths (of length $\dist(i,j)$) from $i$ to $j$.  Thus it is of the form $\sum_{P\in{\mathcal{P}(i,j)}: l(P)=\dist(i,j)} a^P$.  This means the highest-order term in $u_i$ is of degree $|V|-(\dist(i,j)+1) = |V|-1-\dist(i,j)$.
\end{proof}

We now give a formula for the number of coefficients of the input-output equation in the case of either single input or single output.

\begin{thm} [Number of nonzero coefficients] \label{thm:numcoeffs} Let $\cm=(G,\{i_1, i_2, ..., i_{|In|}\},\{j\},L)$ represent a linear compartmental model with $G$ output connectable with at least $|In \cup Out|$ leaks with $In \cup Out \subseteq L$.  There are $|V| + n|V|-\sum_{k} \dist(i_k,j) + m(|V|-1) $ nonzero coefficients where $n = |In - Out|$ and $m=|In \cap Out|$. Now let $\cm=(G,\{i\},\{j_1, j_2, ..., j_{|Out|}\},L)$ represent a linear compartmental model with $G$ output connectable to every output with at least $|In \cup Out|$ leaks with $In \cup Out \subseteq L$.  There are $|V| + n|V|-\sum_{k} \dist(i,j_k)) + m (|V|-1) $ nonzero coefficients where $n = |Out - In|$ and $m=|In \cap Out|$.
\end{thm}

\begin{proof} Assume $\cm=(G,\{i_1, i_2, ..., i_{|In|}\},\{j\},L)$ is a linear compartmental model with $G$ output connectable with at least $|In \cup Out|$ leaks with $In \cup Out \subseteq L$. By Equation (\ref{eq:output-connectable-i-o}), the highest degree term is $|V|$ on the left hand side.  For the right hand side, $|In \cap Out|$ is either $1$ or $0$.  If $|In \cap Out|=1$, then the highest-order term in $u_j$ is of degree $|V|-1$ on the right hand side. The highest-order term in $u_j$ is monic, so there are $|V|-1$ coefficients of terms in $u_j$.  For each $i \in In - Out$, the highest degree term in $u_i$ is of order $|V|-1-\dist(i,j)$ on the right hand side by Lemma \ref{lemma:hot}.  In this case, the highest-order term in $u_i$ is not monic, so there are $|V|-1-\dist(i,j) + 1 = |V|-\dist(i,j)$ coefficients of terms in $u_i$ when $i \neq j$.  Altogether, there are $|V| + n|V|-\sum_{k} \dist(i_k,j)) + m(|V|-1) $ nonzero coefficients where $n = |In - Out|$ and $m=|In \cap Out|$.  

For the case with multiple outputs, let $\cm=(G,\{i\},\{j_1, j_2, ..., j_{|Out|}\},L)$ represent a linear compartmental model with $G$ output connectable to every output with at least $|In \cup Out|$ leaks with $In \cup Out \subseteq L$.  By applying the formula for the case of single output above for each input-output equation, we obtain that there are $|V| + n|V|-\sum_{k} \dist(i,j_k)) + m (|V|-1) $ nonzero coefficients where $n = |Out - In|$ and $m=|In \cap Out|$.

We need only show that the coefficients are nonzero (for a generic choice of parameters).

If there are leaks from every compartment,  Proposition \ref{prop:factor} shows that the coefficients factor through cycles, self-cycles, and paths from input to output.

Now consider the case of removing leaks.  We will be substituting $a_{ii}$ as the negative sum of all outgoing edges when $i \notin L$, but if $i \in L$ then $a_{ii}$ stays the same.  Since $In \cup Out \subseteq L$, we have that every compartment has an outgoing edge or leak, as every vertex has an outgoing edge except for possibly the output vertex by the output connectable assumption. This means the substitution $a_{ii}$ as the negative sum of all outgoing edges and leaks retains the $(i,i)$ entry of $A$ to be nonzero. 

Recall by Proposition \ref{prop:factor}, these coefficients can be factored over cycles, self-cycles, and paths when $L=V$.  Each coefficient, except for the highest order coefficient in $u_{i_k}$ when $i_k \neq j$ (which is a sum of paths from $i_k$ to $j$), must have a term involving a self-cycle.  If the self-cycles in every coefficient are only from leak compartments, we are done.  Otherwise, consider a coefficient that has terms involving self-cycles from non-leak compartments which we must substitute into for the case $L \subset V$.  We want to show that the substitution of the non-leak diagonal terms as the negative sum of all outgoing edges does not cancel every term in that coefficient, so that the coefficients remain nonzero after substitution.  We claim that the substitution of $a_{kk}$ as the negative sum of all outgoing edges for $k \notin L$ cannot create only terms that are products of cycles and paths from input to output.  Since the graph must be output connectable, any cycle formed from the non-leak vertices must connect to the output. In other words, for a chain of vertices in a cycle $k_1, k_2,...,k_l$, one of these vertices must connect to the output via a path from that vertex to the output.  Without loss of generality, assume it is vertex $k_1$.  Thus the substitution of the non-leaks $a_{k_{1}k_{1}}$, $a_{k_{2}k_{2}}$, ...,$a_{k_{l}k_{l}}$ cannot create a single monomial term of the form $\pm a_{k_{1}k_{2}}a_{k_{2}k_{3}}\cdots a_{k_{l}k_{1}}$, but must also create a monomial $\pm a_{k_{1}k_{2}}a_{k_{2}k_{3}}\cdots a_{r k_{1}}$ where we have substituted $a_{k_{1}k_{1}}$ as $-a_{k_{l}k_{1}} - a_{r k_{1}}$ for some vertex $r$ that connects via a path to the output.

This monomial $a_{k_{1}k_{2}}a_{k_{2}k_{3}}\cdots a_{r k_{1}}$ cannot itself be a path from input to output, as the input and output vertices have leaks and thus the corresponding diagonal terms do not get substituted.  Thus, it is not a path from input to output, and thus cannot cancel with any other terms in that coefficient.   
\end{proof}

\begin{rmk} We note that the assumption of at least $|In \cup Out|$ leaks with $In \cup Out \subseteq L$ and $G$ to be output connectable is to prevent the situation where there are no outgoing edges or leaks from a non-leak vertex and thus upon substitution of the diagonal element $a_{ii}$ as the negative sum of all outgoing edges and leaks, it becomes zero.
\end{rmk}

\begin{defn} [Expected number of coefficients] We say a model $\cm=$ \\
$(G,\{i_1, i_2, ..., i_{|In|}\},\{j\},L)$ with $In \cup Out \subseteq L$ has the \textit{expected number of coefficients} if there are 
$|V| + n|V|-\sum_{k} \dist (i_k,j) + m(|V|-1) $ nonzero coefficients in the input-output equations (\ref{eq:i-o-for-general-model}) where $n = |In - Out|$ and $m=|In \cap Out|$. We say a model $\cm=(G,\{i\},\{j_1, j_2, ..., j_{|Out|}\},L)$ with $In \cup Out \subseteq L$ has the \textit{expected number of coefficients} if there are 
$|V| + n|V|-\sum_{k} \dist (i,j_k) + m (|V|-1) $ nonzero coefficients in the input-output equations (\ref{eq:i-o-for-general-model}) where $n = |Out - In|$ and $m=|In \cap Out|$.
\end{defn}

\section{Identifiable path/cycle models} \label{section:idpathcycle}

\begin{defn} We say a model $\cm=(G,In,Out,V)$ has a coefficient map with \textit{expected dimension} if the dimension of the image of the coefficient map is maximal.
\end{defn}

We will be examining a special class of models we call \textit{identifiable path/cycle models}.  This class of models is a generalization of \textit{identifiable cycle models} as defined in \cite{MeshkatSullivantEisenberg}:

\begin{defn} [Identifiable Cycle Models] We say a model $\cm=(G,\{i\},\{i\},V)$ with $G$ strongly connected is an \textit{identifiable cycle model} if all of the independent monomial cycles in the model are locally identifiable.
\end{defn}

\begin{ex} The models $(G, \{1\},\{1\},V)$ and $(H,\{1\},\{1\},V)$ where $G$ corresponds to a chain of exchanges $1 \leftrightarrow 2, 2 \leftrightarrow 3$, etc, and $H$ correspond to a central compartment given by compartment $1$ and exchanges $1 \leftrightarrow 2, 1 \leftrightarrow 3$, etc, are identifiable cycle models due to Theorem 5.13 of \cite{MeshkatSullivant}.  The first model is commonly called a \textit{catenary} model and the second model is called a \textit{mammillary} model \cite{distefano-book}. 
\end{ex}

It was shown in \cite{MeshkatSullivant} that a sufficient condition for a model to be an identifiable cycle model is that the dimension of the image of the coefficient map is $|E|+1$.  We now define the main object of interest in this paper, \textit{identifiable path/cycle models} and spend the rest of this section forming analogous sufficient conditions on the dimension of the image of the coefficient map.

\begin{defn} [Identifiable Path/Cycle Models] \label{defn:pathcycle} We say a model $\cm=(G,In,Out,V)$ is an \textit{identifiable path/cycle model} if all of the independent monomial cycles and monomial paths from input to output in the model are locally identifiable and each parameter is contained in such a cycle or path.
\end{defn}

\begin{rmk} As identifiable path/cycle models require all of the cycles and paths from input to output in the model to be identifiable, it only makes sense to consider models that are \textit{connected and every edge is contained in a cycle or path from input to output}, i.e. strongly input-output connected.  Otherwise, an edge that is not contained in a cycle or path from input to output will not appear in the coefficient map.
\end{rmk}

\begin{ex} [Continuation of Example \ref{ex:continuing}] The model $\cm=(G,\{1\},\{2\},V)$ with $G$ given by the graph \{ $1 \rightarrow 2, 2 \rightarrow 3, 3 \rightarrow 2, 3 \rightarrow 4, 4 \rightarrow 3$ \} is an identifiable path/cycle model, with identifiable paths and cycles given by $a_{11}, a_{22}, a_{33}, a_{44}, a_{21}, a_{23}a_{32}, a_{34}a_{43}$. This can be demonstrated by writing each path and cycle as a function of the coefficients $c_i$, e.g. using Groebner Bases. Further justification will come from Theorem \ref{thm:idpathcycle}.
\end{ex}

For models with leaks from every compartment, the dimension of the image of the coefficient map is bounded above by the number of independent paths and cycles in the graph from Proposition \ref{prop:factor}. We now determine what this number is. We first show that when $G$ is output connectable for the case of single output, there are $|E|+|In \cup Out|$ independent directed paths and undirected cycles.  We then examine the case where $G$ is strongly input-output connected so that the indicator vectors for the independent directed paths and directed cycles in the graph correspond to $0/1$ vectors.  We show that this number of independent paths and cycles is equal to $|E|+ |In \cup Out|$.

We define the $|V|$ by $|E|$ \textit{incidence matrix} $E(G)$ as:

\begin{equation}\label{eq:eg}
  E(G)_{i,(j,k)} = \left\{ 
  \begin{array}{l l l}
    1 & \quad \text{if $i=j$}\\
   -1 & \quad \text{if $i=k$}\\
    0 & \quad \text{otherwise.}\\
  \end{array} \right.
\end{equation}
In other words, $E(G)$ has column vectors corresponding to the edges 
$j\rightarrow{k} \in E$ with a $1$ in the $jth$ row, $-1$ in the $kth$ row, and $0$ 
otherwise.  We define the \textit{indicator vector} of a directed cycle $C$ as the vector $(x_s)_{s \in E}$ such that $x_s = 1$ if $s \in E_C$ and $x_s = 0$ if $s \notin E_C$, where $E_C$ is the set of edges associated to the directed cycle $C$.

We can also define the \textit{indicator vector} of an undirected cycle $C'$ with associated directed cycle $C$ (reversing arrows to all point in the same direction) as the vector $(x_s)_{s \in E}$ such that $x_s = 1$ if $s \in E_C$, $x_s = -1$ if $-s \in E_C$, and $x_s = 0$ if $s \notin E_C$, where $E_C$ is the set of edges associated to the directed cycle $C$ and $-s$ corresponds to an edge $s$ going in the opposite direction. In other words, if $s$ corresponds to $i \to j$, then $-s$ corresponds to $j \to i$.  

The rank of the directed incidence matrix is well-known:

\begin{prop} [Proposition 4.3 of \cite{Biggs}] Let $G$ be a graph with $|V|$ vertices, $|E|$ edges, and $l$ connected components.  Then the rank of $E(G)$ is $|V|-l$.  Thus, the dimension of the kernel of $E(G)$ is $|E|-|V|+l$.
\end{prop}

We state one final result from \cite{MeshkatSullivant, MeshkatSullivantEisenberg}, which shows that the kernel of $E(G)$ can be written in terms of $|E|-|V|+1$ directed cycles when $G$ is strongly connected, thus the indicator vectors are $0/1$ vectors:

\begin{prop} \label{prop:kernelincidence} [Proposition 4 of \cite{MeshkatSullivantEisenberg}] Let $G$ be a strongly connected graph.  Then a set of $|E|-|V|+1$ linearly independent indicator vectors of directed cycles form a basis for the kernel of $E(G)$.  
\end{prop}

In other words, this proposition shows that the space of all undirected cycles can be generated by the space of all directed cycles when $G$ is strongly connected.  We now prove a similar result in terms of cycles and paths from input to output when $G$ is strongly input-output connected.

\begin{prop} \label{prop:undirected} Let $\cm=(G,In,Out,V)$ represent a linear compartmental model with $G$ strongly input-output connected. Then the space of all directed paths and undirected cycles can be generated by the space of all directed paths and directed cycles and vice versa.
\end{prop}

\begin{proof}
Let $B$ have as its columns the indicator vectors of all directed paths and undirected cycles.  We show that, for every undirected cycle, we can add a positive integer multiple of a directed path vector or directed cycle vector to obtain either a directed cycle or directed path from input to output.  Since $G$ is strongly input-output connected, every edge is in either a cycle or path from input to output.  For every edge with a negative entry in the indicator vector of an undirected cycle, that edge either belongs to a cycle or path from input to output.  If it belongs to a path from input to output, one can add a positive multiple of the path to the undirected cycle to achieve only non-negative entries corresponding to a closed path or path from input to output.  If it does not belong to a path from input to output, then it belongs to a directed cycle.  Thus one can add a positive multiple of the directed cycle to the undirected cycle to achieve only non-negative entries corresponding to a closed path or path from input to output.  In either case, this corresponds to a multigraph with the property that the indegree of each vertex equals the outdegree of each vertex except possibly at input and output vertices.  Cycles can be removed so that the result is a cycle or a path from input to output.
\end{proof}

\begin{lemma}\label{lemma:numpathscycles} Let $\cm=(G,In,Out,V)$ represent a linear compartmental model with $G$ output connectable if $|Out|=1$ or $G$ strongly input-output connected otherwise. Then the number of independent undirected cycles and directed paths from input to output is $|E|+|In \cup Out|$.
\end{lemma}

\begin{proof} Let the matrix $B$ have as columns the indicator vectors of the \textit{undirected} cycles and directed paths from input to output vertices.  Since $G$ is either output connectable in the single output case or strongly input-output connected, then $G$ is certainly connected, and thus there are $|E|-|V|+1$ undirected cycles that form a basis for the kernel of $E(G)$ and there must be at least one path from input to output, so $B$ is certainly not the zero matrix.  Form the product $E(G)B$.  If column $k$ of $B$ corresponds to a cycle, then column $k$ of $E(G)B$ will be zero, and if column $k$ of $B$ corresponds to a path from input in $i$ and output in $j$, then column $k$ of $E(G)B$ will have a $1$ in row $i$ and a $-1$ in row $j$. 

Remove the zero columns and duplicate columns (which occur when there is more than one path from an input to an output) and zero rows from this matrix $E(G)B$ and call the resulting matrix $M$.  We claim $M$ is the incidence matrix of the graph where there are $|In \cup Out|$ vertices corresponding to each input/output compartment and there is a directed edge from an input compartment to an output compartment if and only if there is a path from the corresponding input to the corresponding output in the graph $G$.  Call this graph $G_M$.  Note that there are only $|In \cup Out|$ vertices in $G_M$ because we deleted zero rows from the matrix  $E(G)B$ to obtain the matrix $M$, thus deleting vertices that do not correspond to inputs or outputs.  This graph $G_M$ must be connected because we assumed $G$ is output connectable in the single output case and strongly input-output connected otherwise.  Since the rank of the incidence matrix for a connected graph is the number of vertices minus one, this means the rank of $E(G)B$ is $|In \cup Out| - 1$.

Since the rank of $E(G)B$ is equal to the rank of $B$ minus the dimension of the column space of $B$ intersected with the kernel of $E(G)$, which is exactly $|E|-|V|+1$ because $B$ is generated by paths and undirected cycles and a basis for the kernel of $E(G)$ is given by undirected cycles, then this means the rank of $B$ is exactly $|E|-|V|+1+|In \cup Out|-1=|E|-|V|+|In \cup Out|$. Adding the $|V|$ self-cycles, we obtain that the dimension of the path/cycle map is $|V|+|E|-|V|+|In \cup Out|=|E|+|In \cup Out|$. 
\end{proof}

\begin{cor}\label{cor:numpathscycles} Let $\cm=(G,In,Out,V)$ represent a linear compartmental model with $G$ strongly input-output connected.  Then the number of independent directed cycles and directed paths from input to output is $|E|+|In \cup Out|$. 
\end{cor}

\begin{proof} If $|Out|=1$, strongly input-output connected implies output connectable and if $|Out|>1$, we have strongly input-output connected.  The statement follows from Lemma \ref{lemma:numpathscycles} and Proposition \ref{prop:undirected} to achieve a set of $|E|+|In \cup Out|$ independent directed cycles and directed paths from input to output.
\end{proof}

We now show that the dimension of the image of the coefficient map is bounded above by the number of independent paths and cycles.  We will add the important assumption of either $|In|=1$ or $|Out|=1$ so that the number of distinct input-output pairs equals $|In \cup Out|-1$, described in the Remark below.  For $|Out|=1$ we can assume $G$ is strongly input-output connected as stated in Corollary \ref{cor:numpathscycles}, but for the case of $|In|=1$ we will assume $G$ is strongly connected in order to ensure the input-output equations are irreducible as shown in Section 2.  For the special case where $|In|=|Out|=1$ and $In = Out$, we note that strongly input-output connected reduces to strongly connected.  For the special case where $|In|=|Out|=1$ and $In \neq Out$, then strongly input-output connected is sufficient in what follows, i.e. we can take the weaker of the two conditions strongly input-output connected and strongly connected.

\begin{lemma} \label{lemma:dimc} Let $\cm=(G,In,Out,V)$ represent a linear compartmental model.  Assume that either $G$ is strongly input-output connected and $|Out|=1$ or $G$ is strongly connected and $|In|=1$. The dimension of the image of the coefficient map is bounded above by $|E|+|In \cup Out|$.
\end{lemma}

\begin{proof} The coefficient map factors through the cycles, self-cycles, and paths from input to output from Proposition \ref{prop:factor}. By Corollary \ref{cor:numpathscycles}, the number of independent paths and cycles is $|E|+|In \cup Out|$.  Thus the dimension of the image of the coefficient map is bounded above by $|E|+|In \cup Out|$.
\end{proof} 

\begin{rmk} We require $|In|=1$ or $|Out|=1$ so that there are either $|Out|-|In \cap Out|$ or $|In|-|In \cap Out|$ distinct input-output pairs, respectively, which equals $|In \cup Out|-1$, the rank of $E(G)B$ in the proof of Lemma \ref{lemma:numpathscycles}.  Example \ref{ex:singleinsingleout} demonstrates this.
\end{rmk}

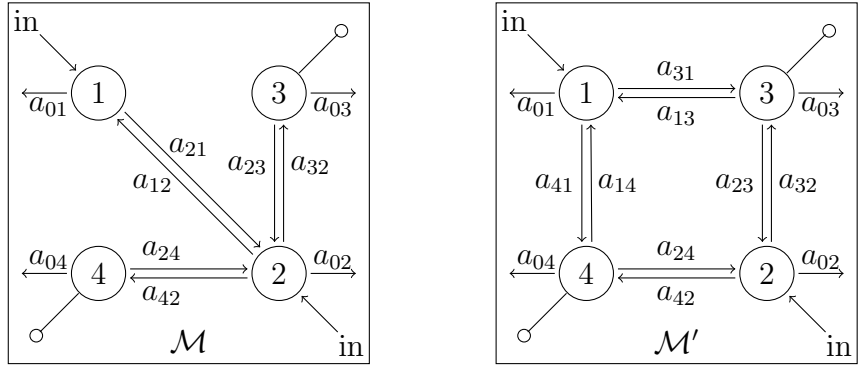
\begin{figure}[H]
\begin{center}
	\begin{tikzpicture}[scale=1.2]
 	\draw (0,0) circle (0.3);
 	\draw (2,0) circle (0.3);
 	\draw (2,-2) circle (0.3);
 	\draw (0,-2) circle (0.3);
    	\node[] at (0, 0) {1};
    	\node[] at (2, 0) {3};
    	\node[] at (2, -2) {$2$};
    	\node[] at (0, -2) {$4$};
 	 \draw[<-] (2.05,-.35) -- (2.05,-1.65);
 	 \draw[->] (1.95,-.35) -- (1.95,-1.65);
	 \draw[->] (0.35, -1.95) -- (1.65, -1.95);
	 \draw[<-] (0.35, -2.05) -- (1.65, -2.05);
	 \draw[->] (.29,-.21) -- (1.79,-1.71);
	 \draw[<-] (.21,-.29) -- (1.71,-1.79);
   	 \node[] at (2.35, -.8) {$a_{32}$};
	 \node[] at (1.65,-.8) {$a_{23}$};
	 \node[] at (.7,-1.75) {$a_{24}$};
	 \node[] at (.7,-2.25) {$a_{42}$};
	 \node[] at (.6,-1) {$a_{12}$};
	 \node[] at (1,-.6) {$a_{21}$};
	\draw (2.69,.69) circle (0.07);	
	 \draw[-] (2.65, .65 ) -- (2.22, .22);	
	\draw (-.69,-2.69) circle (0.07);	
	 \draw[-] (-.65, -2.65 ) -- (-.22, -2.22);	
	 \draw[->] (-.65, .65) -- (-.25, .25);	
   	 \node[] at (-.8,.8) {in};
	 \draw[->] (2.65, -2.65) -- (2.25, -2.25);	
   	 \node[] at (2.8,-2.8) {in};

	 \draw[->] (-0.35, -2) -- (-.85, -2);	
   	 \node[] at (-.55, -1.85) { $a_{04}$};
	 \draw[->] (2.35, -2) -- (2.85, -2);	
   	 \node[] at (2.6, -1.85) {$a_{02}$};
	 \draw[->] (-0.35, 0) -- (-.85, 0);	
   	 \node[] at (-.55, -0.15) {$a_{01}$};
	 \draw[->] (2.35, 0) -- (2.85, 0);	
   	 \node[] at (2.6, -0.15) {$a_{03}$};

%
%
\draw (-1,-3) rectangle (3, 1);
    	\node[] at (1, -2.75) {$\cm$};
%

\end{tikzpicture}
\hspace*{.5in}
	\begin{tikzpicture}[scale=1.2]
 	\draw (0,0) circle (0.3);
 	\draw (2,0) circle (0.3);
 	\draw (2,-2) circle (0.3);
 	\draw (0,-2) circle (0.3);
    	\node[] at (0, 0) {1};
    	\node[] at (2, 0) {3};
    	\node[] at (2, -2) {$2$};
    	\node[] at (0, -2) {$4$};
 	 \draw[<-] (2.05,-.35) -- (2.05,-1.65);
 	 \draw[->] (1.95,-.35) -- (1.95,-1.65);
	 \draw[->] (0.35, -1.95) -- (1.65, -1.95);
	 \draw[<-] (0.35, -2.05) -- (1.65, -2.05);
	 \draw[->] (0.35, .05) -- (1.65, .05);
	 \draw[<-] (0.35, -.05) -- (1.65, -.05);
 	 \draw[<-] (.05,-.35) -- (.06,-1.65);
 	 \draw[->] (-.05,-.35) -- (-.05,-1.65);

   	 \node[] at (2.35, -1) {$a_{32}$};
	 \node[] at (1.65,-1) {$a_{23}$};
	 \node[] at (1,-1.75) {$a_{24}$};
	 \node[] at (1,-2.25) {$a_{42}$};
	 \node[] at (1,.25) {$a_{31}$};
	 \node[] at (1,-.25) {$a_{13}$};
	 \node[] at (.35,-1) {$a_{14}$};
	 \node[] at (-.35,-1) {$a_{41}$};

	\draw (2.69,.69) circle (0.07);	
	 \draw[-] (2.65, .65 ) -- (2.22, .22);	
	\draw (-.69,-2.69) circle (0.07);	
	 \draw[-] (-.65, -2.65 ) -- (-.22, -2.22);	
	 \draw[->] (-.65, .65) -- (-.25, .25);	
   	 \node[] at (-.8,.8) {in};
	 \draw[->] (2.65, -2.65) -- (2.25, -2.25);	
   	 \node[] at (2.8,-2.8) {in};

	 \draw[->] (-0.35, -2) -- (-.85, -2);	
   	 \node[] at (-.55, -1.85) { $a_{04}$};
	 \draw[->] (2.35, -2) -- (2.85, -2);	
   	 \node[] at (2.6, -1.85) {$a_{02}$};
	 \draw[->] (-0.35, 0) -- (-.85, 0);	
   	 \node[] at (-.55, -0.15) {$a_{01}$};
	 \draw[->] (2.35, 0) -- (2.85, 0);	
   	 \node[] at (2.6, -0.15) {$a_{03}$};

%
%
\draw (-1,-3) rectangle (3, 1);
    	\node[] at (1, -2.75) {$\cm'$};
%

\end{tikzpicture}
\end{center}
\caption{The models $\cm$ and $\cm'$ from Example \ref{ex:singleinsingleout}.}
\end{figure}

\begin{ex} \label{ex:singleinsingleout}
The model $\cm=(G,\{1,2\},\{3,4\},V)$ where $G$ is given by \{ $1 \rightarrow 2, 2 \rightarrow 1, 2 \rightarrow 3, 3 \rightarrow 2, 2 \rightarrow 4, 4 \rightarrow 2$ \} has $|E|+|In \cup Out|=10$ independent paths and cycles but the coefficient map factors over $11$ paths and cycles given by $a_{12}a_{21}$, $a_{23}a_{32}$, $a_{24}a_{42}$, $a_{32}$, $a_{42}$, $a_{21}a_{32}$, $a_{21}a_{42}$, $a_{11}$, $a_{22}$, $a_{33}$, $a_{44}$.  The problem here is that there are only $10$ parameters, but we are attempting to factor over $11$ paths and cycles. This is why we require $|In|=1$ or $|Out|=1$.  However, this is a sufficient condition but not a necessary condition, as the model $(G',\{1,2\},\{3,4\},V)$ where $G'$ is given by \{ $1 \rightarrow 3, 3 \rightarrow 1, 1 \rightarrow 4, 4 \rightarrow 1, 2 \rightarrow 3, 3 \rightarrow 2, 2 \rightarrow 4, 4 \rightarrow 2$ \} has $|E|+|In \cup Out|=12$ independent paths and cycles given by $a_{31}$, $a_{41}$, $a_{32}$, $a_{42}$, $a_{13}a_{31}$, $a_{14}a_{41}$, $a_{23}a_{32}$, $a_{24}a_{42}$, $a_{11}$, $a_{22}$, $a_{33}$, $a_{44}$ and the coefficient map factors over these as well. 
\end{ex}

\begin{rmk} Notice that in Lemma \ref{lemma:dimc}, we assume that $G$ is either strongly input-output connected or strongly connected.  We have just shown that the expected dimension in this case is $|E|+|In \cup Out|$.  But a natural question that arises is, what if we do not assume this connectedness condition on $G$?  Clearly, we still have that the coefficient map factors through cycles and paths.  In this case there will be \textit{at most} $|E|+|In \cup Out|$ independent cycles and paths that appear in the coefficient map, i.e. the coefficient map may factor over fewer than $|E|+|In \cup Out|$ independent paths and cycles. See Section \ref{section:othered}.  
\end{rmk}

This gives us the following theorem:

\begin{thm} \label{thm:idpathcycle} Let $\cm=(G,In,Out,V)$ represent a linear compartmental model with either $G$ strongly input-output connected and $|Out|=1$ or $G$ strongly connected and $|In|=1$.  If the image of the coefficient map has dimension $|E|+|In \cup Out|$, then the model is an identifiable path/cycle model.  
\end{thm}

\begin{proof} By Lemma \ref{lemma:dimc}, the dimension of the image of the coefficient map is bounded above by $|E| + |In \cup Out|$, which is also the number of independent paths and cycles.  Recall a function $f$ is locally identifiable if there is a finitely multivalued function $\phi: \rr^{|E|+|In \cup Out|} \rightarrow \rr$ such that $\phi \circ c = f$.  Let $\pi: \rr^{|E|+|V|} \rightarrow \rr^{|E|+|In \cup Out|}$ be the path/cycle map from Equation \ref{eq:pi}.  Since $c: \rr^{|E|+|V|} \rightarrow \rr^{|E|+|In \cup Out|}$ factors over paths and cycles, then there exists a function $\psi: \rr^{|E|+|In \cup Out|} \rightarrow \rr^{|E|+|In \cup Out|}$ as defined in Equation \ref{eq:psi} such that $c = \psi \circ \pi$.  If the dimension of the image of the coefficient map is precisely $|E|+|In \cup Out|$, then this function $\psi$ is locally invertible with $\psi^{-1}=\phi$ and thus $\pi = \phi \circ c$.  Thus the paths and cycles are identifiable.
\end{proof}

\begin{ex} [Continuation of Example \ref{ex:continuing}] The model $\cm=(G,\{1\},\{2\},V)$ from Example \ref{ex:continuing} can be shown to have dimension of the image of the coefficient map equal to $|E|+|In \cup Out|=5+2=7$.  Thus there are $7$ identifiable paths and cycles given by the monomials $a_{11}, a_{22}, a_{33}, a_{44}, a_{21}, a_{23}a_{32}, a_{34}a_{43}$. 
\end{ex}

\subsection{Necessary condition for number of edges and the edge inequality}

\begin{prop} Let $\cm=(G,In,Out,V)$ represent a linear compartmental model with either $G$ strongly input-output connected and $|Out|=1$ or $G$ strongly connected and $|In|=1$.  If $\cm=(G,In,Out,V)$ is an identifiable path/cycle model, then $|E| + |In \cup Out| \leq$ the expected number of coefficients.  
\end{prop}

\begin{proof} We have that the dimension of the image of the coefficient map is bounded above by the expected number of coefficients, as it cannot exceed the number of coefficients.  Thus $|E|+|In \cup Out| \leq $ the expected number of coefficients.
\end{proof}

\begin{defn} [Edge inequality] We say that a model has a number of edges given by the \textit{edge inequality} if the number of edges $|E|$ satisfies $|E|+|In \cup Out| \leq$ the expected number of coefficients.
\end{defn}

We can also show that the property of being strongly input-output connected is a necessary condition for having expected dimension in the case of maximal number of edges.

\begin{prop} \label{prop:maxedgesneccond} Let $\cm=(G,In,Out,V)$ with $G$ output connectable and $|Out|=1$ represent a linear compartmental model for which the
edge inequality is an equality with expected dimension $|E|+|In \cup Out|$.  Then the graph must be strongly input-output connected.
\end{prop} 

\begin{proof}
Assume the coefficient map has expected dimension $|E|+|In \cup Out|$ with the maximal number of edges.  This means the expected dimension is the number of coefficients.  If the graph is not connected with every edge in a cycle or path from input to output, then there are parameters that do not appear in the coefficient map, and thus the coefficient map factors over fewer than $|E|+|In \cup Out|$ independent paths and cycles.  But this contradicts having expected dimension $|E|+|In \cup Out|$.
\end{proof} 

\begin{rmk} If there are fewer edges, we can still achieve expected dimension without this condition of strongly input-output connected. In other words, being strongly input-output connected is a sufficient but not necessary condition to achieve $|E|+|In \cup Out|$ independent cycles and paths in the coefficient map.  See Section \ref{section:othered}.
\end{rmk}

\subsection{Obtaining identifiability by removing leaks}

In this section, we show that removing all leaks except leaks from input/output compartments results in identifiability, much like the results in \cite{MeshkatSullivantEisenberg}.  We will follow the same proof.  Recall that for an identifiable path/cycle model, we have the coefficient map $c:\rr^{|V|+|E|} \rightarrow \rr^k$.  Let $\pi:\rr^{|V|+|E|} \rightarrow \rr^{|E|+|In \cup Out|}$ be the path/cycle map from Equation \ref{eq:pi}, that is $\pi(A(G))=(a^P : P$ is a cycle or path from input to output of $G)$.  Then Proposition \ref{prop:factor} tells us that $c$ factors through $\pi$ without loss of dimension.  Thus $c = \psi \circ \pi$ where $\psi$ is defined as in Equation \ref{eq:psi} and the dimension of the image of $c$ equals the dimension of the image of $\pi$.

Passing from a model $(G,In,Out,V)$ to a model $(G,In,Out,Leak)$ such
that $|Leak|=|In \cup Out|$
amounts to restricting the parameter space $\rr^{|V| + |E|}$
to a linear subspace $\Lambda \subseteq  \rr^{|V| + |E|}$ of dimension $|E| + |Leak|$
and we would like the image of $\Lambda$ under the coefficient map $c$
to have dimension $|E| + |Leak|$.  Since  $c$ factors through the path/cycle map $\pi$ it suffices to prove
that the image of $\Lambda$ under $\pi$ has dimension $|E| + |Leak|$.

\begin{lemma} \label {lemma:algindset} Let $G = (V,E)$ be a directed graph with corresponding identifiable path/cycle model $(G,In,Out,V)$. Assume that either $G$ is strongly input-output connected and $|Out|=1$ or $G$ is strongly connected and $|In|=1$. Consider a model $(G,In,Out,L)$ where $In \cup Out \subseteq L$.  Let $\pi : \rr^{|V|+|E|} \rightarrow \rr^{|E|+|In \cup Out|}$ denote the path/cycle map. 
Let $\Lambda \subseteq \rr^{|V| + |E|}$ be the linear space satisfying
$$
\Lambda = \{ \mathcal{A} \in \rr^{|V| + |E|}: a_{ii} =  - \sum_{j,j \neq i} a_{ji} \mbox{ for all }
i \notin L \}.
$$
Then the dimension of the image of $\Lambda$ under the map $\pi$ is $|E| + |In \cup Out|$. 
\end{lemma}

\begin{proof} Since $\Lambda$ is a linear space, we just consider the natural map from $\rr^{|E|+|L|} \to \rr^{|E|+|In \cup Out|}$ which maps to the path/cycle space.  To show that the dimension of the image of this map is correct, we consider the Jacobian of this map and show that it has full rank.

Note that the rows corresponding to the $|L|$ self-cycles are linearly independent so we focus on the $|E| + |In \cup Out| - |L|$ by $|E|$ submatrix ignoring those rows and columns, which we will call $J$.  Arrange the matrix so that the first $|E|-|V|+|In \cup Out|$ rows correspond to the paths and cycles of $G$ and the last $|V|-|L|$ rows correspond to the non-leak diagonal elements.  Let the first $|E|-|V|+|In \cup Out|$ rows be called $A$ and the last $|V|-|L|$ rows be called $B$.  Clearly the rows of $A$ are linearly independent by Lemma \ref{lemma:dimc}.  The rows of $B$ are linearly independent since they are in triangular form since each involves distinct parameters.  
To show that the full set of $|E| + |In \cup Out| - |L|$ rows are linearly independent, we need only show that the row space of $A$ and the row space of $B$ intersect only in the origin.

To prove that $J$ generically has maximal possible rank, it is enough to show that there is some point where the evaluation of $J$ at said point yields the maximal rank.  We choose the point where we set all the edge parameters $a_{ij}=1$ for all $j\to i \in E$.  This specialization yields that the row space of $A$ is exactly the path/cycle space of the graph $G$, i.e. all of the weightings on the edges of the graph where the indegree equals the outdegree of every vertex in a cycle and every vertex except the first and last in a path from input to output.  Also, we have that the matrix $B$ which has dimension $(|V|-|L|) \times (|E|)$, which consists of the rows corresponding to the vertices in $V\setminus L$, and the (negated) row corresponding to vertex $i$ has a one for an edge $i' \to j'$ if and only if $i=i'$, with all other entries zero.

Since $A$ spans the path/cycle space of $G$, each element in the row space of $A$ corresponds to a weighting on the edges of $G$ where the total weight of all incoming edges at a vertex $i$ equals the total weight of all outgoing edges at vertex $i$ except at input or output vertices $In \cup Out$.  On the other hand, we claim that the only vector in the row span of $B$ with the same property is the zero vector.  To show this, let $b_i$ be the row vector associated to some vertex $i$. Note that a vector in the row span of $B$ will have zero weight on any of the outgoing edges of vertices in $In \cup Out$.  

In order for the indegree to equal the outdegree, we would need to include a $b_j$ with an edge pointing toward vertex $i$.  Continuing in this way, we can only stop when we have included an input or output vertex since the indegree need not equal the outdegree for those vertices.  However, this contradicts the fact that a vector in the row span of $B$ will have zero weight on any of the outgoing edges of vertices in $In \cup Out$. 
\end{proof}

\begin{thm} [Removing Leaks] \label{thm:removeleaks} Let $\cm=(G,In,Out,V)$ represent a linear compartmental model. Assume that either $G$ is strongly input-output connected and $|Out|=1$ or $G$ is strongly connected and $|In|=1$. Assume it is an identifiable path/cycle model. Then, the corresponding model $\wcm=(G,In,Out,L)$ where $In \cup Out \subseteq L$ for any such $L$ has expected dimension. In particular, if $L = In \cup Out$, then $\wcm$ is locally identifiable.
\end{thm}

\begin{proof}
By Lemma \ref{lemma:algindset} and the comments preceding it
we know that the image of the restricted parameter space
under the path/cycle map $\pi$ has dimension $|E| + |In \cup Out|$, which 
is equal to the dimension of the image of the full parameter
space under the path/cycle map.  Since, for an identifiable path/cycle model, the
dimension of the image of the coefficient map $c$ is $|E| + |In \cup Out|$,
this must be the same for the restricted model. In particular, if $|L|=|In \cup Out|$, then the model
has $|E| + |In \cup Out|$ parameters, hence it is locally identifiable.
\end{proof}

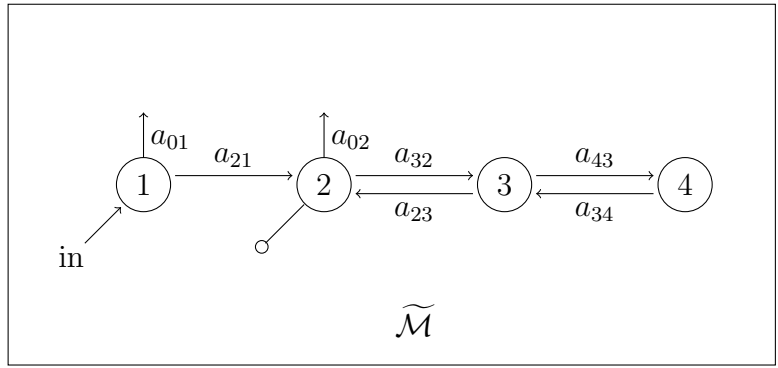
\begin{figure}[H]
\begin{center}
	\begin{tikzpicture}[scale=1.2]
 	\draw (0,0) circle (0.3);
 	\draw (2,0) circle (0.3);
 	\draw (4,0) circle (0.3);
 	\draw (6,0) circle (0.3);
    	\node[] at (0, 0) {1};
    	\node[] at (2, 0) {2};
    	\node[] at (4, 0) {$3$};
    	\node[] at (6, 0) {$4$};
	 \draw[->] (0.35, .1) -- (1.65, .1);
	 \draw[->] (2.35,.1) -- (3.65,.1);
	 \draw[<-] (2.35,-.1) -- (3.65,-.1);
	 \draw[->] (4.35,.1) -- (5.65,.1);
	 \draw[<-] (4.35,-.1) -- (5.65,-.1);
   	 \node[] at (1, 0.3) {$a_{21}$};
	\node[] at (3,0.3) {$a_{32}$};
	\node[] at (3,-.3) {$a_{23}$};
	\node[] at (5,.3) {$a_{43}$};
	\node[] at (5,-.3) {$a_{34}$};
	\draw (1.31,-.69) circle (0.07);	
	 \draw[-] (1.35, -.65 ) -- (1.78, -.22);	
	 \draw[->] (-.65, -.65) -- (-.25, -.25);	
   	 \node[] at (-.8,-.8) {in};
	 \draw[->] (0,.3) -- (0, .8);	
   	 \node[] at (.3, .5) {$a_{01}$};
	 \draw[->] (2,.3) -- (2, .8);	
   	 \node[] at (2.3, .5) {$a_{02}$};
\draw (-1.5,-2) rectangle (7, 2);
    	\node[] at (3, -1.5) {$\widetilde{\cm}$};
    	

\end{tikzpicture}
\end{center}
\caption{Graph for model corresponding to $\widetilde{\cm}$ in Example \ref{ex:removecontinuing}.}
\end{figure}

\begin{ex} [Continuation of Example \ref{ex:continuing}] \label{ex:removecontinuing} The model $\wcm=(G,\{1\},\{2\},\{1,2\})$ obtained from the model in Example \ref{ex:continuing} by removing two leaks and leaving the leaks in the input and output compartments is locally identifiable. 
\end{ex}

\begin{rmk} We note that, while $L=In \cup Out$ is sufficient in Theorem \ref{thm:removeleaks}, it is certainly not necessary, as there are other possible configurations of $|L|=|In \cup Out|$ leaks that also result in identifiability. The next example demonstrates this.
\end{rmk}

\begin{figure}[H]
\begin{center}
	\begin{tikzpicture}[scale=1.2]
 	\draw (0,0) circle (0.3);
 	\draw (2,0) circle (0.3);
 	\draw (2,-2) circle (0.3);
 	\draw (0,-2) circle (0.3);
    	\node[] at (0, 0) {1};
    	\node[] at (2, 0) {2};
    	\node[] at (2, -2) {$3$};
    	\node[] at (0, -2) {$4$};
	 \draw[->] (0.35, 0) -- (1.65, 0);
 	 \draw[<-] (.3,-.2) -- (1.8,-1.7);
 	 \draw[->] (.2,-.3) -- (1.7,-1.8);
 	 \draw[<-] (0.07,-.35) -- (0.07,-1.65);
 	 \draw[->] (-0.07,-.35) -- (-0.07,-1.65);
   	 \node[] at (1, 0.15) {$a_{21}$};
	\node[] at (.34,-1) {$a_{14}$};
	\node[] at (-.33,-1) {$a_{41}$};
	\node[] at (1.5,-1) {$a_{13}$};
	\node[] at (1,-1.5) {$a_{31}$};
	\draw (2.69,.69) circle (0.07);	
	 \draw[-] (2.65, .65 ) -- (2.22, .22);	
	 \draw[->] (-.65, .65) -- (-.25, .25);	
   	 \node[] at (-.8,.8) {in};
	 \draw[->] (-0.35, -2) -- (-.85, -2);	
   	 \node[] at (-.55, -2.15) {$a_{04}$};
	 \draw[->] (2.35, 0) -- (2.85, 0);	
   	 \node[] at (2.6, -.15) {$a_{02}$};

\draw (-1,-3) rectangle (3, 1);
    	\node[] at (1, -2.7) {$\cm$ for $L =\{2,4\}$};
%

\end{tikzpicture}
\hspace*{.3in} 
\begin{tikzpicture}[scale=1.2]
 	\draw (0,0) circle (0.3);
 	\draw (2,0) circle (0.3);
 	\draw (2,-2) circle (0.3);
 	\draw (0,-2) circle (0.3);
    	\node[] at (0, 0) {1};
    	\node[] at (2, 0) {2};
    	\node[] at (2, -2) {$3$};
    	\node[] at (0, -2) {$4$};
	 \draw[->] (0.35, 0) -- (1.65, 0);
 	 \draw[<-] (.3,-.2) -- (1.8,-1.7);
 	 \draw[->] (.2,-.3) -- (1.7,-1.8);
 	 \draw[<-] (0,-.35) -- (0,-1.65);
 	 \draw[<-] (.25,-1.75) -- (1.75,-.25); 
   	 \node[] at (1, 0.15) {$a_{21}$};
	\node[] at (-.3,-1) {$a_{14}$};
	\node[] at (1.8,-1.3) {$a_{13}$};
	\node[] at (1.3,-1.8) {$a_{31}$};
	\node[] at (1.8,-.7) {$a_{42}$};
	\draw (2.69,.69) circle (0.07);	
	 \draw[-] (2.65, .65 ) -- (2.22, .22);	
	 \draw[->] (-.65, .65) -- (-.25, .25);	
   	 \node[] at (-.8,.8) {in};
	 \draw[->] (-0.35, -2) -- (-.85, -2);	
   	 \node[] at (-.55, -2.15) {$a_{04}$};
	 \draw[->] (2.35, -2) -- (2.85, -2);	
   	 \node[] at (2.6, -2.15) {$a_{03}$};

\draw (-1,-3) rectangle (3, 1);
    	\node[] at (1, -2.7) {$\cm'$ for $L =\{3,4\}$};
%

\end{tikzpicture}
\end{center}
\caption{On the left is the graph corresponding to $\cm$ with leak set $L=\{2,4\}$, and on the right is the graph corresponding to $\cm'$ with leak set $L=\{3,4\}$ from Example \ref{ex:leakplacement}.}
\end{figure}
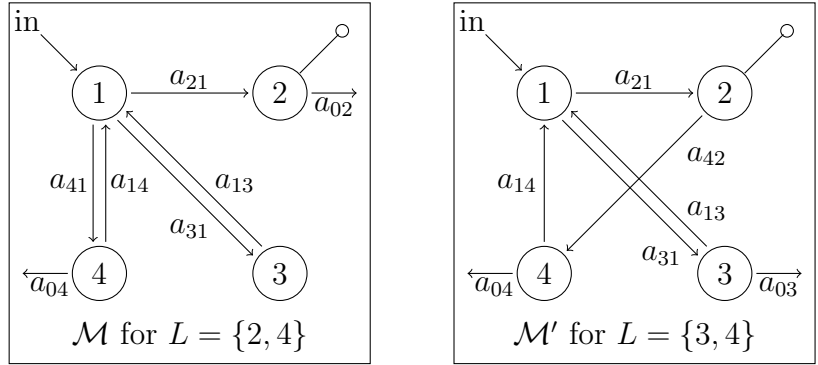

\begin{ex} \label{ex:leakplacement} Consider the model $\cm=(G,\{1\},\{2\},L)$ where $|L|=|In \cup Out|=2$ and $G$ is given by the edges \{ $1 \rightarrow 2, 1 \rightarrow 3, 3 \rightarrow 1, 1 \rightarrow 4, 4 \rightarrow 1$ \}.  The identifiable models are the ones where $L = \{2,4\}, \{2,3\},\{1,2\}$ and the unidentifiable models have $L=\{3,4\},\{1,4\},\{1,3\}$.  

We note that while $2 \in L$ appears to be sufficient for identifiability in this model, we can consider another model given by $\cm'=(G',\{1\},\{2\},L)$ where $|L|=|In \cup Out|=2$ and $G'$ is given by the edges \{ $1 \rightarrow 2, 3 \rightarrow 1, 4 \rightarrow 1, 1 \rightarrow 3, 2 \rightarrow 4$ \}.  The identifiable models are the ones where $L = \{3,4\}, \{2,3\},\{1,4\}, \{1,2\}$ and the unidentifiable models have $L=\{2,4\},\{1,3\}$. This shows the pattern of identifiability depends on the graph structure itself and not just the placement of inputs and outputs. 
\end{ex}

\subsection{Sufficient condition for identifiable path/cycle model} 

We now give a sufficient condition for a model to be an identifiable path/cycle model with 1 input and 1 output.  This sufficient condition is analogous to the sufficient condition from \cite{MeshkatSullivant} of inductively strongly connected for models with input and output in the same compartment.  In fact, Theorem \ref{thm:almost_isc} reduces to Theorem 5.13 of \cite{MeshkatSullivant} if the input and output compartments are the same.

\begin{thm} \label{thm:almost_isc} Let $\cm=(G,\{i\},\{j\},V)$ represent a linear compartmental model with $G$ strongly input-output connected and $|E| = 2|V|-(\dist(i,j)+2)$.  If $\cm=(G,\{i\},\{j\},V)$ has no path from compartment $j$ to compartment $i$ but becomes inductively strongly connected if an edge from compartment $j$ to compartment $i$ is added, then $\cm$ is an identifiable path/cycle model.  
\end{thm}

Before we prove Theorem \ref{thm:almost_isc}, we define a graph structure which will be useful in the proof.

\begin{defn} [Definition 5.6 from \cite{MeshkatSullivant}] \label{defn:chain}
A \textit{chain of cycles} is a graph $H$ which consists of a sequence of directed cycles that are attached to each other in a chain by joining at the vertices.
\end{defn}

\begin{thm}\label{thm:add_node}
Suppose $\mathcal{M}'=(G',\{i\},\{j\},V)$ represents a linear compartmental model with $G'$ strongly input-output connected and $|E|=2|V|-(\dist(i,j) +2)$.  Suppose too that $\mathcal{M}'$ has expected dimension and also that $\mathcal{M}'$ has no path from $j$ to $i$ and becomes inductively strongly connected if the edge from $j$ to $i$ is added.  Then, if $G$ is a new graph obtained from $G'$ by adding a vertex $n$ and two edges $k\to n$ and $n\to l$ such that $G$ has a chain of cycles containing either $i$ and $n$ or $j$ and $n$, then the model $\mathcal{M}=(G,\{i\},\{j\},V \cup \{n\})$ has expected dimension.
\end{thm}

Recall that we can induce a \textit{weight order} on a polynomial ring $\kk[x_1,\ldots , x_n]$ for some weight vector $\omega \in \qq^n$ where the weight of a monomial $x_1^{\alpha_1} \cdots x_n^{\alpha_n}$ is $\omega \cdot \alpha$ where $\alpha = (\alpha_1, \ldots , \alpha_n)$.  We can then define the \textit{initial forms} of a polynomial $f$ as $\inw(f)$ to be the sum of all terms of $f$ whose monomial has the highest weight with respect to said $\omega$.  

Now, if we define the coefficient map associated to the graph $G$ as $\phi_G \colon \rr^{|V|+|E|} \to \rr^{k}$, then we can also consider the pull-back of said map defined as $\phi_G^* \colon \kk [c,d] \to \kk[a]$ where $c,d$ correspond to the coefficients of the left and right-hand side of the input-output equation respectively, and $a$ corresponds to the parameters found in compartmental matrix $A$.  Now define $\phi_{G,\omega}$ to be the initial parameterization defined as the parameterization with pullback $\phi^*_{G,\omega}$ where $\phi^*_{G,\omega} (f) = \inw(f)$ for a given weight $\omega$.

\begin{lemma}[Corollary 5.9, \cite{MeshkatSullivant}]\label{lemma:images}
Let $\phi^* \colon \kk[x] \to \kk[y]$ be a $\kk$-algebra homomorphism and $\omega \in \qq^m$ a weight vector, then
\[
\dim(\im \phi_\omega ) \leq \dim( \im \phi ).
\]
\end{lemma}

We will use Lemma \ref{lemma:images} in the following way.  We want to compute the dimension of the image of a polynomial parametrization $\phi$.  We know for other reasons an upper bound $d$ on this dimension.  We have a weight vector $\omega$ where we can compute the dimension of the image of the polynomial parametrization $\phi_\omega$, and we show it is equal to $d$.  Then, by Lemma \ref{lemma:images}, we know that the dimension of the image of $\phi$ must be $d$.

\begin{proof}[Proof of Theorem \ref{thm:add_node}]
Suppose $\mathcal{M}'=(G',\{i\},\{j\},V)$ is a linear compartmental model with expected dimension such that $G'$ is strongly input-output connected, $|E|=2|V|-(\dist(i,j)+2)$.  Also suppose that if we add the edge from $j$ to $i$, the new graph becomes inductively strongly connected.  Note that by Theorem \ref{thm:numcoeffs}, the input-output equation of the model $\mathcal{M}'$ has $2|V|-\dist(i,j)$ nonzero, non-monic coefficients.  Let $|V|=n-1$ and $m=|E|$.

Define $\phi_G \colon \rr^{|V|+|E|} \to \rr^{2|V|-\dist(i,j)}$, to be the coefficient map associated to a graph $G$ as above with corresponding pull-back $\phi^*_G$.  Choose weight $\omega$ as follows:
\[
\omega_{uv} = \begin{cases}
0 & \text{if } (u,v) = (n,n) \\
\frac{1}{2} & \text{if } (u,v)=(n,k) \text{ or } (l,n) \\
1 & \rm{otherwise}.
\end{cases}
\]

Recall that for each coefficient, the corresponding polynomial function is homogeneous in terms of the parameters.  Also, recall that the left-hand side coefficients are generated by cycles of the corresponding graph, while the right-hand side coefficients are generated by products of cycles of the corresponding graph along with paths from the input to output.

Applying this weight to the polynomial coefficients has the effect of removing any monomial containing a cycle which is incident to compartment $n$, in all coefficients except for the lowest order terms in both $c$ and $d$.  Note that in the case of $c_n$ and $d_{n-1}$, each of the monomials in the sum will have a cycle incident to $n$, meaning that each of them has the same weight.  

More explicitly, in terms of the pull-back maps we have in all cases that
\begin{align*}
\phi^*_{G,\omega}(c_i) &= \phi^*_{G'}(c_i) \quad i=1, \ldots , n-1 \\
\phi^*_{G,\omega}(d_i) &= \phi^*_{G'}(d_i) \quad i=\dist(i,j), \ldots , n-2 \\
\phi^*_{G,\omega}(c_n) &= \phi^*_G(c_n) \\
\phi^*_{G,\omega}(d_{n-1}) &= \phi^*_G(d_{n-1}).
\end{align*}

Thus, $\phi_{G,\omega}$ agrees with $\phi_{G'}$ everywhere except for the highest order coefficients on either side of the input-output equation, in which case $\phi_{G,\omega}$ matches $\phi_G$.  This implies that the Jacobian matrix corresponding to $\phi_{G,\omega}$ defined as $J(\phi_{G,\omega})$ (whose generic rank yields the dimension of the image of the map), has the form
\[
J(\phi_{G,\omega} ) = \begin{pmatrix}
J (\phi_{G'} )& 0 \\ * & C
\end{pmatrix}
\]

where $J(\phi_{G'})$ is the $(2n-(\dist(i,j)+2)) \times (n+m-3)$ Jacobian matrix of $\phi_{G'}$ and $C$ is the $2\times 3$ matrix
\[
C = \begin{pmatrix}
\frac{\partial c_n }{\partial a_{nn}} & \frac{\partial c_n }{\partial a_{ln}} & \frac{\partial c_n}{\partial	a_{nk}} \\
\frac{\partial d_{n-1} }{\partial a_{nn} } & \frac{\partial d_{n-1} }{\partial a_{ln} } & \frac{\partial d_{n-1} }{\partial a_{nk}}
\end{pmatrix}
\]
where $l$ and $k$ are the nodes to which the added node $n$ has an edge to and from respectively.

Note that we assume that the model corresponding to $G'$ has expected dimension, hence $J(\phi_{G'})$ has rank $2(n-1)-\dist(i,j)$.  Since $J(\phi_G)$ is lower block triangular, to show that it has rank $2n-\dist(i,j)$, we need only show that $C$ has generic rank 2.

Let $H$ be a chain of cycles in $G$ defined as $s_2, \ldots , s_t$ in order such that $s_2$ is a cycle containing either the input or the output and $s_t$ is the cycle containing the node $n$.  Also, define $s_1$ to be one of the shortest paths from $i$ to $j$.

Now we will choose entries for the matrix $A$ such that the matrix $C$ has rank 2.   First, let all diagonal elements of $A$ be $1$, i.e. $a_{kk}=1$ for all $k=1,\ldots , n$.  Also, let $a_{uv} = 0$ for all edges $v \to u \not\in H$.  For all edges in $H$, for each cycle $s_i$, choose the edge weights so that the product of edges' weights is equal to $(-1)^{\ell(s_i)-1}$, that is so that the product of the edges in the cycle is equivalent to the sign of the cycle.  For $s_1$, choose edge weights so that the product of the edges' weights is also equal to $(-1)^{\ell(s_i)-1}.$

First, consider the entry $\frac{\partial c_n}{\partial a_{nn}}$.  The only nonzero monomials appearing here will arise from taking products of the cycles $s_2,\ldots , s_{t-1}$, since the cycle $s_t$ cannot be involved, as we only consider the elements of the sum of $c_n$ with $a_{nn}$ as a factor.  Also, $s_1$ is not a cycle, hence cannot be part of any of the $c_i$.  Since each cycle touches its two neighboring cycles, and no other cycles, and in the expansion we expand over all products of nontouching cycles that cover all $n$ vertices, we get that the number of monomials will equal the number of subsets of $\{2,\ldots , t-1\}$ with no adjacent elements.  By Lemma \ref{lemma:fibo}, this is exactly $F_{t}$.

Now consider the entry $\frac{\partial d_{n-1}}{ \partial a_{nn}}$, which will arise from taking products of $s_1,s_3,\ldots , s_{t-1}$, since we must have $s_1$, hence cannot have $s_2$, and by similar reasoning above cannot have $s_t$.  Thus, we get that the number of monomials are the number of nonadjacent subsets of $\{s_3,\ldots, s_{t-1}\}$, hence we get the Fibonacci number $F_{t-1}$.

In the case of the entry $\frac{\partial c_n}{\partial a_{ln}}$ or equivalently $\frac{\partial c_n}{\partial a_{nk}}$, we must use the cycle $s_t$ prohibiting us from using $s_{t-1}$, and again must not use $s_1$.  This yields that the number of monomials is the number of nonadjacent subsets of $s_{2},\ldots , s_{t-2}$, i.e. the Fibonacci number $F_{t-1}$.

Finally, when considering entry $\frac{\partial d_{n-1} }{\partial a_{ln}}$ or equivalently $\frac{\partial d_{n-1} }{\partial a_{nk}}$, we must use $s_1$ and cycles $s_t$, hence cannot use cycles $s_2$ or $s_{t-1}$.  This means that the number of monomials will be the number of nonadjacent subsets of $s_3,\ldots , s_{t-2}$, i.e. the Fibonacci number $F_{t-2}$.

Thus, the submatrix $C$ will have the form 
\[C = \begin{pmatrix}
F_t & F_{t-1} & F_{t-1} \\
F_{t-1} & F_{t-2} & F_{t-2}
\end{pmatrix}.
\]

The classical identity of Fibonacci number $F_{t}F_{t-2} - F_{t-1}^2 = (-1)^{t-1}$ yields that this matrix has full rank. hence, the Jacobian of $\phi^*_{G,\omega}$ has full rank.

Note that the upper bound for the number of coefficients of the input-output equation, i.e. the upper bound on the dimension of the image of the coefficient map is $2n-\dist(i,j)$ via Theorem \ref{thm:numcoeffs}.  Thus, because $\dim(\im(\phi^*_{G,\omega})) \leq \dim(\im(\phi^*_G))$ by Lemma \ref{lemma:images}, and $\dim(\im(\phi^*_G))$ is bounded above by $2n-\dist(i,j)$, we have that $\dim(\im(\phi^*_G))=2n-\dist(i,j)$ as desired.
\end{proof}

\begin{lemma}\label{lemma:fibo}
The number of subsets $S$ of $\{1,2,\ldots , n\}$ such that $S$ contains no pair of adjacent numbers is the $n+2$-nd Fibonacci number, $F_{n+2}$ which satisfies the recurrence $F_0=0$, $F_1=1$, and $F_{n+1}=F_n+F_{n-1}$.
\end{lemma}

We can now prove Theorem \ref{thm:almost_isc}.

\begin{proof} [Proof of Theorem \ref{thm:almost_isc}]
By Theorem \ref{thm:add_node} and the inductive nature of inductively strongly connected graphs, it suffices to show that every inductively strongly connected graph beginning with a cycle between $i$ and $j$ has a chain of cycles containing the vertices $i$ and $n$ (or, analogously, a chain of cycles containing the vertices $j$ and $n$).  

We prove this by induction on $n$.  Since $G$ is inductively strongly connected if the edge from $j$ to $i$ is added, there is a nontrivial cycle $c$ that passes through the vertex $n$.  If $c$ contains $i$, we are done.  Otherwise, let $q$ be the smallest vertex appearing in $c$, and let $G'$ be the induced subgraph on $\{i,j,\ldots,q\}$.  By induction, $G'$ has a chain of cycles containing $i$ and $q$.  Attaching $c$ to $H$ gives a chain of cycles in $G$ containing $i$ and $n$.  A similar argument can be applied to give a chain of cycles in $G$ containing $j$ and $n$.
\end{proof}

Theorem \ref{thm:almost_isc} is not only useful as a sufficient condition for an identifiable path/cycle model, but it is also useful as a means to start with an identifiable path/cycle model and then remove leaks to obtain identifiability:

\begin{cor} \label{cor:almost_isc_id} Let $\cm=(G,\{i\},\{j\},V)$ represent a linear compartmental model with $G$ strongly input-output connected and $|E| = 2|V|-(\dist(i,j)+2)$.  If $\cm=(G,\{i\},\{j\},V)$ has no path from compartment $j$ to compartment $i$ but becomes inductively strongly connected if the edge from compartment $j$ to compartment $i$ is added, then $\cm'=$ \\
$(G,\{i\},\{j\},\{i,j\})$ is locally identifiable, i.e. removing all but two leaks in the input/output compartments.
\end{cor} 

\begin{proof} This follows from Theorem \ref{thm:almost_isc} and Theorem \ref{thm:removeleaks}.
\end{proof}

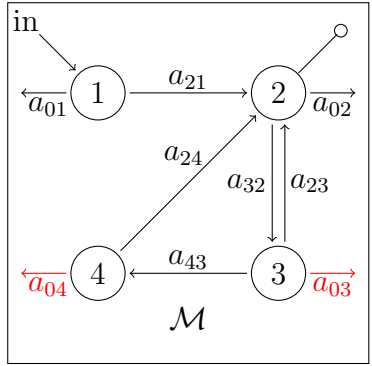
\begin{figure}[H]
\begin{center}
	\begin{tikzpicture}[scale=1.2]
 	\draw (0,0) circle (0.3);
 	\draw (2,0) circle (0.3);
 	\draw (2,-2) circle (0.3);
 	\draw (0,-2) circle (0.3);
    	\node[] at (0, 0) {1};
    	\node[] at (2, 0) {2};
    	\node[] at (2, -2) {$3$};
    	\node[] at (0, -2) {$4$};
	 \draw[->] (0.35, 0) -- (1.65, 0);
 	 \draw[->] (.25,-1.75) -- (1.75,-.25); 
 	 \draw[<-] (2.07,-.35) -- (2.07,-1.65);
 	 \draw[->] (1.93,-.35) -- (1.93,-1.65);
 	 \draw[<-] (0.35,-2) -- (1.65,-2);
   	 \node[] at (1, 0.15) {$a_{21}$};
   	 \node[] at (1, -1.85) {$a_{43}$};
   	 \node[] at (.95, -.7) {$a_{24}$};
   	 \node[] at (2.35, -1) {$a_{23}$};
	 \node[] at (1.65,-1) {$a_{32}$};
	\draw (2.69,.69) circle (0.07);	
	 \draw[-] (2.65, .65 ) -- (2.22, .22);	
	 \draw[->] (-.65, .65) -- (-.25, .25);	
   	 \node[] at (-.8,.8) {in};
	 \draw[->,red] (-0.35, -2) -- (-.85, -2);	
   	 \node[] at (-.55, -2.15) { \color{red}$a_{04}$};
	 \draw[->,red] (2.35, -2) -- (2.85, -2);	
   	 \node[] at (2.6, -2.15) {\color{red} $a_{03}$};
	 \draw[->] (-0.35, 0) -- (-.85, 0);	
   	 \node[] at (-.55, -0.15) {$a_{01}$};
	 \draw[->] (2.35, 0) -- (2.85, 0);	
   	 \node[] at (2.6, -0.15) {$a_{02}$};

%
%
\draw (-1,-3) rectangle (3, 1);
    	\node[] at (1, -2.5) {$\cm$};
%

\end{tikzpicture}
\end{center}
\caption{The model $\cm$ corresponds to the above graph with all four leaks, while the graph $\cm'$ has the same graph with only the black leaks, that is leaks in compartments 1 and 2, all from Example \ref{ex:isc}.}
\end{figure}

\begin{ex} \label{ex:isc} The model $\cm=(G,\{1\},\{2\},V)$ with $G$ given by the edges $\{ 1 \rightarrow 2, 2 \rightarrow 3, 3 \rightarrow 4, 4 \rightarrow 2, 3 \rightarrow 2 \}$ is an identifiable path/cycle model by Theorem \ref{thm:almost_isc}.  Thus, the model $\cm'=(G,\{1\},\{2\},\{1,2\})$ where we remove all but two leaks from the input/output compartments is locally identifiable.
\end{ex}

We also note that, as cycles with input/output in the same compartment were shown to have expected dimension in \cite{MeshkatSullivant} (see Proposition 5.4), paths from input to output can be shown to have expected dimension as well.

\begin{prop} \label{prop:path} Let $\cm=(G,\{1\},\{|V|\},V)$ be a linear compartmental model with $G$ given by a path from input $1$ to output $|V|$ with $|V|-1$ edges.  Then $\cm$ is an identifiable path/cycle model and the model $\wcm=(G,\{1\},\{|V|\},\{1,|V|\})$ is locally identifiable.
\end{prop} 

\begin{proof} Let $n=|V|$. Assume $\cm=(G,\{1\},\{n\},V)$ is a linear compartmental model with $G$ given by a path from input $1$ to output $n$ with $n-1$ edges.  Recall the coefficients on the left hand side of the input-output equation are given by the characteristic polynomial of $A$, which is:

$$(\lambda - a_{11})(\lambda - a_{22})\cdots(\lambda - a_{nn})$$

Since the roots of a polynomial can be determined from its coeffiicents, then all of $a_{11},a_{22},\ldots,a_{nn}$ are locally identifiable.  Since the degree of the highest-order term on the right hand side of the input-output equation is $n-1-\dist(1,n)$ by Lemma \ref{lemma:hot}, this reduces to zero so the right hand side is $a_{n,n-1}\cdots a_{32}a_{21} u_n$. Thus the monomial path $a_{n,n-1}\cdots a_{32}a_{21}$ is identifiable.  This means the dimension of the image of the coefficient map is $|E|+|In \cup Out|=n-1+2=n+1$ which is the number of paths and cycles, thus the model is an identifiable path/cycle model.  By Theorem \ref{thm:removeleaks}, the model $\wcm=(G,\{1\},\{n\},\{1,n\})$ is locally identifiable.
\end{proof}

In \cite{MeshkatSullivant}, it was shown in Proposition 5.5 that if a model $\cm=(G,\{1\},\{1\},V)$ has expected dimension, then the model $\cm=(G',\{1'\},\{1'\},V)$ also has expected dimension, where $G'$ is the new graph obtained from $G$ by adding a new vertex $1'$ and an exchange $1 \rightarrow 1', 1' \rightarrow 1$ and making $1'$ the new input-output node.  We show an analogous result now:

\begin{prop} \label{prop:addvertex} Let $\cm=(G,\{1\},\{j\},V)$ be a linear compartmental model that has expected dimension where $|V|=n$ and $j \neq 1$.  Let  $G'$ be a new graph obtained from $G$ by adding a set of new vertices $n+1, n+2, \ldots, n+k$ and a set of edges $n+1 \rightarrow 1, n+2 \rightarrow n+1, \ldots, n+k \rightarrow n+(k-1)$ and making $n+k$ the new input node. Then the model $\wcm=(G',\{n+k\},\{j\},V \cup \{n+1, n+2, \ldots, n+k\})$ also has expected dimension. 
\end{prop} 

\begin{proof} Let $A$ be the full matrix associated to the graph $G'$ where the first $k$ rows and $k$ columns correspond to the added path from compartment $n+k$ to compartment $1$, $A_{k}$ be the matrix where the first $k$ rows and first $k$ columns have been deleted (and, hence associated to the graph $G$), and let $E_G$ be the edges of the graph $G$.  We assume that the dimension of the image of the map $c$ associated to the model $\cm=(G,\{1\},\{j\},V)$ is $|E_G|+|In \cup Out| = |E_G|+2$, and we want to show that for the model $\wcm=(G',\{n+k\},\{j\},V \cup \{n+1, n+2, \ldots, n+k\})$ we get $|E_{G'}|+2=|E_G|+k+2$, as we are adding $k$ new edges.  

The input-output equation for the model $\wcm=(G',\{n+k\},\{j\},V \cup \{n+1, n+2, \ldots , n+k\})$ is:

$$\det(\partial{I}-A)y_j = \det(\partial{I}-A)_{1,{j+k}} u_{n+k} $$

where $\det(\partial{I}-A)=(\partial - a_{n+1,n+1})\cdots(\partial - a_{n+k,n+k})\det(\partial{I}-A_{k})$ and $\det(\partial{I}-A)_{1,{j+k}}= a_{1,n+1} a_{n+2,n+1} \cdots a_{n+k,n+k-1} \det(\partial{I}-A_{k})_{1j}$. The input-output equation for the model $\cm=(G,\{1\},\{j\},V)$ is:  

$$\det(\partial{I}-A_{k})y_j = \det(\partial{I}-A_{k})_{1j} u_{1} $$

For notational ease, write $\partial$ as $\lambda$, $|V|=n$, $p =(a_{n+1,n+1}, \ldots, a_{n+k,n+k})$, and $q=a_{1,n+1} a_{n+2,n+1} \cdots a_{n+k,n+k-1}$. Note that $p$ is the vector of new self-cycles and $q$ can be interpreted as the added monomial path from compartment $n+k$ to compartment $1$.  We can write $\det(\lambda{I}-A_{k})$ as:
$$
\lambda^n + c_1 \lambda^{n-1} + \ldots + c_{n-1} \lambda + c_n
$$
 and we can write $\det(\lambda{I}-A_{k})_{1j}$ as 
$$
d_1\lambda^{n-1} + d_2 \lambda^{n-2} + \ldots + d_{n-1} \lambda + d_{n}.
$$
  Thus $\det(\lambda{I}-A)=(\lambda-a_{n+1,n+1}) \cdots (\lambda-a_{n+k,n+k})\det(\lambda{I}-A_{k})$ can be written as (up to a minus sign):
	\begin{multline} \label{multiline:lefthandside}
	\lambda^{n+k}+(c_1-S_1(p))\lambda^{n+k-1}+
	(c_2 -c_1 S_1(p) + S_2(p)) \lambda^{n+k-2} + 
	(c_3-c_2 S_1(p)+c_1 S_2(p) -S_3(p)) \lambda^{n+k-3} \\
		+ \ldots + (c_k - c_{k-1} S_1(p) + c_{k-2} S_2(p) - \ldots -S_k(p)) \lambda^{n} \\
		+ \ldots + (c_n -c_{n-1} S_1(p) + c_{n-2} S_2(p) - \ldots - c_{n-k} S_k (p)) \lambda^{k} \\
		+ (-c_n S_{1}(p) + \ldots + c_{n-k+2} S_{k-1}(p) - c_{n-k+1} S_k(p)) \lambda^{k-1} \\
		+ \ldots +
		(-c_n S_{k-2}(p) + c_{n-1} S_{k-1}(p) - c_{n-2} S_k(p)) \lambda^{2}
		+ (c_n S_{k-1}(p) - c_{n-1} S_k(p))\lambda - c_n S_k(p)
		\end{multline}
	where $S_1(p), \ldots, S_k(p)$ are the $k$ elementary symmetric polynomials in the parameter vector $p$. Here we assumed $n>k$, but an analogous formula follows for the case of $n \leq k$.
		
	We will refer to the non-constant coefficients of $\det(\lambda{I}-A)$ as $C_1, \ldots, C_{n+k}$.  Note that these are by assumption identifiable.	
		Likewise, 
		\begin{align}\label{align:righthandside}
		\det(\lambda{I}-A)_{1,{j+k}}=q d_1\lambda^{n-1} + q d_2 \lambda^{n-2} + \ldots + q d_{n-1} \lambda + q d_{n}.
		\end{align}
	We will refer to the coefficients of $\det(\lambda{I}-A)_{1,{j+k}}$ as $D_1, \ldots, D_n$.  Note that these are by assumption identifiable.
		We must now show that if the mapping given by $(c_1, \ldots, c_n, d_1, \ldots , d_n)$ has expected dimension, then the new mapping $(C_1, \ldots, C_{n+k}, D_1, \ldots, D_n)$ also has expected dimension.  Since the parameters in $p$ are roots of the polynomial $\det(\lambda{I}-A)$, then this means these parameters can be written in terms of the coefficients $(C_1, \ldots, C_{n+k})$, which are identifiable, and thus the parameters in $p$ are identifiable.  This means each of the elementary symmetric polynomials $S_1(p), \ldots, S_k(p)$ are identifiable. Since $C_1$ and $S_1(p)$ are identifiable from Equation \ref{multiline:lefthandside}, then $c_1$ can be recovered from the first coefficient from Equation \ref{multiline:lefthandside}.  Likewise, since $C_2$, $c_1$, $S_1(p)$, and $S_2(p)$ are identifiable, then $c_2$ can be recovered from the second coefficient of Equation \ref{multiline:lefthandside}.  Continuing in this fashion, we can recover $c_1, \ldots, c_n$, i.e. $c_1, \ldots, c_n$ are identifiable.  This means the dimension of the image of $(C_1, \ldots, C_{n+k})$ is $k$ more than the dimension of the image of $(c_1, \ldots, c_n)$.  Since the coefficients of Equation \ref{align:righthandside} are just the coefficients $d_1, \ldots, d_n$ scaled by $q$, which contains disjoint parameters from the parameters in the coefficients $c_1, \ldots, c_n, d_1, \ldots, d_n$, then the dimension of the image of $(D_1, \ldots, D_n)$ is the same as the dimension of the image of $(d_1, \ldots, d_n)$.  The parameters in $p$ do not appear in $(D_1, \ldots, D_n)$, thus combining the maps $(C_1, \ldots, C_{n+k})$ and $(D_1, \ldots, D_n)$ we have that the dimension of the image of the new map $(C_1, \ldots, C_{n+k}, D_1, \ldots, D_n)$ must be $k$ more than the dimension of the image of $(c_1, \ldots, c_n, d_1, \ldots , d_n)$ due to the identifiability of the parameters in $p$.  Thus the model $\wcm=(G',\{n+k\},\{j\},V \cup \{n+1, n+2, \ldots, n+k\})$ has dimension of the image of the coefficient map equal to $|E_G|+k+2$, i.e. $k$ more than $\cm=(G,\{1\},\{j\},V)$.
		

\end{proof}

\section{Classification of all identifiable models that are strongly input-output connected with 1 output or strongly connected with 1 input and leaks in input/output compartments} \label{section:classification}

The following Theorem \ref{thm:addleaks} gives necessary conditions for strongly input-output connected models with 1 output or strongly connected models with 1 input with leaks in input/output compartments to be identifiable, namely that they must be identifiable path/cycle models when all the leaks are added to the model.

\subsection{Necessary conditions for identifiability}

\begin{thm} [Adding Leaks] \label{thm:addleaks}
Let $\cm=(G,In,Out,L)$ represent a linear compartmental model with $|L|=|In \cup Out|$ and either $G$ strongly input-output connected and $|Out|=1$ or $G$ strongly connected and $|In|=1$ which we assume has expected dimension, i.e. has dimension of the image of the coefficient map equal to $|E|+|In \cup Out|$. Then, the corresponding model with an additional leak $\wcm=(G,In,Out,L \cup \{k\})$ also has expected dimension.  Thus the model $\wcm=(G,In,Out,Leak)$ where $L \subseteq Leak$ and $|Leak| \leq |V|$ also has expected dimension.  
\end{thm}

\begin{proof}
Suppose $\cm=(G,In,Out,L)$ is a linear compartmental model with either $G$ strongly input-output connected and $|Out|=1$ or $G$ strongly conencted and $|In|=1$ and $|L| = |In \cup Out|$ which we assume has expected dimension, i.e. has dimension of the image of the coefficient map equal to $|E|+|In \cup Out|$.

Note that because we assume that $\cm$ has expected dimension and $|L|$ leaks, this implies that the Jacobian of the coefficient map has the expected number of coefficients as the number of rows and $(|E|+|In \cup Out|)$ columns with full rank $|E|+|In \cup Out|$.  Note too that the addition of the $|V|-|L|$ parameters from the leaks being added to the model $\cm$ will not increase the number of coefficients in the resulting input-output equation, as the number of coefficients is the maximal amount by Theorem \ref{thm:numcoeffs}.

Therefore, the Jacobian of the coefficient map of the model $\wcm=(G,In,Out,V)$ generated by forcing every compartment in $\cm$ to have a leak, has the same number of rows but now $(|E|+|V|)$ columns.  The dimension of the image of the coefficient map is bounded above by the number of cycles and paths when there are $|V|$ leaks, which is $|E| + |In \cup Out| $.  Thus adding $|V|-|L|$ leaks to a $|L|$-leak model cannot increase the dimension of the image of the coefficient map above $|E| + |In \cup Out|$ if it has already achieved that dimension with $|L|$ leaks.  

Note then that if we consider the specialization generated by substituting zero for every added leak, and consider the submatrix of said Jacobian with expected number of coefficients as the number of rows and $(|E|+|In \cup Out|)$ columns generated by the $(|E|+|In \cup Out|)$ columns corresponding to the edges and leaks in $L$, we have exactly the Jacobian of the coefficient map of $\cm$, which we know is full rank.  Therefore, we have that the Jacobian of the coefficient map of $\wcm$ is also full rank, implying that the model has expected dimension.  

The same argument applies for adding any number of leaks up to $|V|$ total leaks.
\end{proof}

\begin{rmk} We note that while Theorem \ref{thm:removeleaks} and Theorem \ref{thm:addleaks} assume opposite operations of adding or subtracting leaks, we have the condition in Theorem \ref{thm:removeleaks} that $In \cup Out \subseteq Leak$, while in Theorem \ref{thm:addleaks} only $|Leak|=|In \cup Out|$ is assumed, i.e. only the number and not the placement of leaks matters.
\end{rmk}

\subsection{Necessary and sufficient conditions for identifiability}

Combining Theorem \ref{thm:addleaks} with Theorem \ref{thm:removeleaks}, we now come to the main result of this section and obtain the following necessary and sufficient conditions for identifiable models:

\begin{cor} Let $\cm=(G,\{i\},\{i\},\{k\})$ represent a linear compartmental model with $G$ strongly connected and let $\wcm=(G,\{i\},\{i\},V)$ be the corresponding model with a leak in every compartment.  $\cm$ is locally identifiable if and only if $\wcm$ is an identifiable cycle model.  
\end{cor}

\begin{proof} This follows from combining Theorem \ref{thm:addleaks} and Theorem 1 of \cite{MeshkatSullivantEisenberg} (also written as Theorem \ref{thm:mse}).
\end{proof}

\begin{cor} \label{cor:necandsuff} Let $\cm=(G,In,Out,L)$ represent a linear compartmental model with and $L=In \cup Out$ and assume that either $G$ is strongly input-output connected and $|Out|=1$ or $G$ is strongly connected and $|In|=1$. Let $\wcm=(G,In,Out,V)$ be the corresponding model with a leak in every compartment.  $\cm$ is locally identifiable if and only if $\wcm$ is an identifiable path/cycle model.  
\end{cor}

\begin{proof} This follow from combining Theorem \ref{thm:addleaks} and Theorem \ref{thm:removeleaks}.
\end{proof}

\begin{rmk} Corollary \ref{cor:necandsuff} gives us a complete classification of all identifiable models that are strongly input-output connected with 1 output or strongly connected with 1 input and leaks in input/output compartments.  We note that this class of models has the very special property of being \textit{dimension-preserving} when leaks are added or subtracted from non-input/output compartments, up to a point.  To demonstrate that this special dimension-preserving property when removing leaks is not always the case, we revisit Example \ref{ex:leakplacement}.
\end{rmk} 

\begin{ex} Recall Example \ref{ex:leakplacement} where we had the model $\cm=(G,\{1\},\{2\},L)$ where $|L|=|In \cup Out|=2$ and $G$ is given by the edges $\{ 1 \rightarrow 2, 1 \rightarrow 3, 3 \rightarrow 1, 1 \rightarrow 4, 4 \rightarrow 1 \}$.  The identifiable models are the ones where $L = \{2,4\}, \{2,3\},\{1,2\}$ and the unidentifiable models have $L=\{3,4\},\{1,4\},\{1,3\}$, so removing leaks from output compartments is not dimension-preserving for this example.
\end{ex} 

In the next section, we give a conjecture about removing leaks from non-input/output compartments in the general output connectable case for models with one output.

\section{Other expected dimension results} \label{section:othered}

We first show that there are at most $|E|+|In \cup Out|$ independent paths and cycles appearing in the coefficient map $c$ if we relax the condition of strongly input-output connected to \textit{output connectable} instead:

\begin{prop}\label{prop:atmost} Let $\cm=(G,In,Out,V)$ represent a linear compartmental model with $G$ output connectable.  Assume that $|Out|=1$.  Then there are at most $|E|+|In \cup Out|$ independent paths and cycles in the coefficient map $c$.
\end{prop}

\begin{proof} If $G$ is strongly input-output connected (or strongly connected) then we have already shown in Lemma \ref{lemma:dimc} that the coefficient map factors through $|E|+|In \cup Out|$ independent paths and cycles.  If $G$ is output connectable but not strongly input-output connected, then there may be fewer than $|E|+|In \cup Out|$ independent directed paths and directed cycles because there are $|E|-|V|+|In \cup Out|$ independent directed paths and undirected cycles by Lemma \ref{lemma:numpathscycles}. Since the coefficient map factors over the directed paths and directed cycles, then there are at most $|E|+|In \cup Out|$ independent paths and cycles in the coefficient map $c$. 
\end{proof}

This means the expected dimension is now the number of independent directed paths and directed cycles, which may be less than $|E|+|In \cup Out|$.

We can relax the connectedness conditions in Theorem \ref{thm:addleaks}  to \textit{output connectable} instead and still obtain statements about expected dimension, although now the models with a full set of leaks are \textbf{not} identifiable path/cycle models. 

\begin{thm} \label{thm:othered}
Let $\cm=(G,In,Out,L)$ represent a linear compartmental model with $G$ output connectable and $|L| = |In \cup Out|$ and $|Out|=1$ which we assume has dimension of the image of the coefficient map equal to $|E|+|In \cup Out|$.  Then, the corresponding model with a leak in every compartment $\wcm=(G,In,Out,V)$ also has dimension of the image of the coefficient map equal to $|E|+|In \cup Out|$.
\end{thm}

\begin{proof} The proofs mirrors the one in Theorem \ref{thm:addleaks}.
\end{proof}

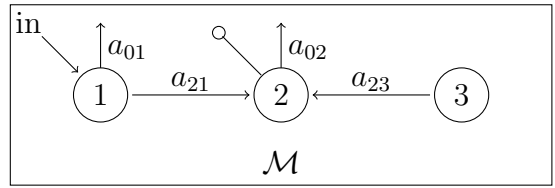
\begin{figure}[H]
\begin{center}
	\begin{tikzpicture}[scale=1.2]
 	\draw (0,0) circle (0.3);
 	\draw (2,0) circle (0.3);
	\draw (4,0) circle (0.3);
    	\node[] at (0, 0) {1};
    	\node[] at (2, 0) {2};
    	\node[] at (4, 0) {$3$};
	 \draw[->] (0.35, 0) -- (1.65, 0);
 	 \draw[<-] (2.35,0) -- (3.65,0);
   	 \node[] at (1, 0.15) {$a_{21}$};
	 \node[] at (3,0.15) {$a_{23}$};
	\draw (1.31,.69) circle (0.07);	
	\draw[-] (1.35, .65 ) -- (1.78, .22);	
	 \draw[->] (-.65, .65) -- (-.25, .25);	
   	 \node[] at (-.8,.8) {in};
	 \draw[->] (0,.3) -- (0, .8);	
   	 \node[] at (.3, .5) {$a_{01}$};
   	 \draw[->] (2,.3) -- (2,.8);
   	 \node[] at (2.3,.5) {$a_{02}$};
%
%
\draw (-1,-1) rectangle (5, 1);
    	\node[] at (2, -.75) {$\cm$};
%

\end{tikzpicture}
\end{center}
\caption{The graph corresponding to model $\cm$ from Example \ref{ex:notidpathcycle}.}
\end{figure}

\begin{ex} \label{ex:notidpathcycle} The model $\cm=(G,\{1\},\{2\},\{1,2\})$ where $G$ is the graph given by $\{ 1 \rightarrow 2, 3 \rightarrow 2 \}$ is output connectable and has dimension of the image of the coefficient map equal to $|E|+2=4$, thus it is locally identifiable.  By Theorem \ref{thm:othered}, the model $\wcm=(G,\{1\},\{2\},V)$ also has dimension of the image of the coefficient map equal to $|E|+2=4$.  Thus the identifiable functions are $a_{11}, a_{22}, a_{33}, a_{21}$.  Note that it is \textit{not} an identifiable path/cycle model because the parameter $a_{23}$ does not appear in the coefficient map (as it is not strongly input-output connected).  
\end{ex}

This result shows that if a model has its dimension of the image of the coefficient map is equal to $|E|+|In \cup Out|$, then adding leaks alone maintains the dimension of the image of the coefficient map.  This result is perhaps more useful for its contrapositive, i.e. if a model with leaks from every compartment does not have  dimension $|E|+|In \cup Out|$ for $c$, then no amount of removing leaks up to a certain point ($|L|=|In \cup Out|$) can attain identifiability.  

\begin{cor} \label{cor:notid} Let $\wcm=(G,In,Out,V)$ represent a linear compartmental model with $G$ output connectable and $|Out|=1$ which does \textbf{not} have the dimension of the image of the coefficient map equal to $|E|+|In \cup Out|$. Then, the corresponding model $\cm=(G,In,Out,L)$ with $|L|=|In \cup Out|$ is \textbf{not} locally identifiable.
\end{cor}

\begin{figure}[H]
\begin{center}
	\begin{tikzpicture}[scale=1.2]
 	\draw (0,0) circle (0.3);
 	\draw (2,0) circle (0.3);
	\draw (1,-1.5) circle (0.3);
    	\node[] at (0, 0) {1};
    	\node[] at (2, 0) {2};
    	\node[] at (1, -1.5) {$3$};
	 \draw[->] (0.35, 0) -- (1.65, 0);
 	 \draw[->] (1.25,-1.25) -- (1.85,-.35);
 	 \draw[->] (.75,-1.25) -- (.15,-.35);
   	 \node[] at (1, 0.15) {$a_{21}$};
	 \node[] at (1.75,-1) {$a_{23}$};
	 \node[] at (.25,-1) {$a_{13}$};
	\draw (1.31,.69) circle (0.07);	
	\draw[-] (1.35, .65 ) -- (1.78, .22);	
	 \draw[->] (-.65, .65) -- (-.25, .25);	
   	 \node[] at (-.8,.8) {in};
	 \draw[->] (0,.3) -- (0, .8);	
   	 \node[] at (.3, .5) {$a_{01}$};
   	 \draw[->] (2,.3) -- (2,.8);
   	 \node[] at (2.3,.5) {$a_{02}$};
   	 \draw[->] (1,-1.8) -- (1,-2.3);
   	 \node[] at (1.3,-2) {$a_{03}$};
%
%
\draw (-1,-3) rectangle (2.75, 1);
    	\node[] at (2, -2.5) {$\cm$};
%

\end{tikzpicture}
\end{center}
\caption{The graph corresponding to model $\cm$ from Example \ref{ex:notid}. }
\end{figure}
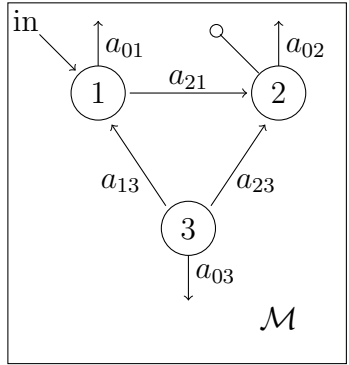

\begin{ex} \label{ex:notid} The model $\cm=(G,\{1\},\{2\},V)$ where $G$ is the graph given by $\{ 1 \rightarrow 2, 3 \rightarrow 2, 3 \rightarrow 1 \}$ is output connectable and has dimension of the image of the coefficient map \textit{not} equal to $|E|+2=5$, but equal to $4$ instead.  By Corollary \ref{cor:notid}, the model $\wcm=(G,\{1\},\{2\},L)$ where $|L|=2$ is thus \textit{not} locally identifiable.
\end{ex}

These results also give us insight into models that have dimension of the image of the coefficient map equal to $|E|+|In \cup Out|$ but are not strongly input-output connected.  We can show that the self-cycles are always identifiable:

\begin{thm} \label{thm:selfcycles} Let $\cm=(G,In,Out,V)$ represent a linear compartmental model with $G$ output connectable and $|Out|=1$.  If $\cm$ has dimension of the image of the coefficient map equal to $|E|+|In \cup Out|$, then the self-cycles $a_{11},...,a_{nn}$ are locally identifiable.
\end{thm}

\begin{proof} Since the coefficient map always factors over $a_{11}, ...,a_{nn}$, this means the self-cycles are locally identifiable.
\end{proof}

Finally, we give a conjecture on removing leaks from non-input/output compartments for output connectable models and prove this conjecture in a special case:

\begin{conj} \label{conj:dimpreserving}
Let $\cm=(G,In,Out,V)$ represent a linear compartmental model. Assume that $G$ is output connectable and $|Out|=1$.  Assume the dimension of the image of the coefficient map is $k$.  Then, the corresponding model $\wcm=(G,In,Out,L)$ where $In \cup Out \subseteq L$ also has dimension of the image of its coefficient map as $k$.
\end{conj}

In other words, we conjecture that this property of being dimension-preserving applies to all output connectable models.  However, if the dimension to begin with is not maximal, the dimension-preserving property will not lead to identifiability.  We give a proof of this conjecture in the special case where the dimension of the image of the coefficient map is $|E|+|In \cup Out|$:

\begin{thm} \label{thm:removeleaksoutputconnectable} Let $\cm=(G,In,Out,V)$ represent a linear compartmental model with $G$ output connectable and $|Out|=1$.  If $\cm$ has dimension of the image of the coefficient map equal to $|E|+|In \cup Out|$, then the the corresponding model $\wcm=(G,In,Out,L)$ where $In \cup Out \subseteq L$ also has dimension of the image of its coefficient map as $|E|+|In \cup Out|$. In particular, if $L = In \cup Out$, then $\wcm$ is locally identifiable.
\end{thm}

To prove Theorem \ref{thm:removeleaksoutputconnectable}, we give a variation of Lemma \ref{lemma:algindset} and then a variation of the proof of Theorem \ref{thm:removeleaks}.  

\begin{lemma} \label {lemma:algindsetoutputconnectable} Let $G = (V,E)$ be a directed graph with corresponding model $(G,In,Out,V)$. Assume that $G$ is output connectable and $|Out|=1$. Consider a model $(G,In,Out,L)$ where $In \cup Out \subseteq L$.  Let $\pi : \rr^{|V|+|E|} \rightarrow \rr^{|E|+|In \cup Out|}$ denote the path/cycle map. 
Let $\Lambda \subseteq \rr^{|V| + |E|}$ be the linear space satisfying
$$
\Lambda = \{ \mathcal{A} \in \rr^{|V| + |E|}: a_{ii} =  - \sum_{j,j \neq i} a_{ji} \mbox{ for all }
i \notin L \}.
$$
If the dimension of the image of $\pi$ is $|E|+|In \cup Out|$, then the dimension of the image of $\Lambda$ under the map $\pi$ is $|E| + |In \cup Out|$. 
\end{lemma}

\begin{proof} Removing the assumption of strongly input-output connected from Lemma \ref{lemma:algindset} means that we cannot guarantee there are $|E|+|In \cup Out|$ independent directed paths and directed cycles. However, if we assume the dimension of the image of $\pi$ is $|E|+|In \cup Out|$, then the rest of the proof follows that of Lemma \ref{lemma:algindset}.
\end{proof}

\begin{proof} [Proof of Theorem \ref{thm:removeleaksoutputconnectable}]  
By Lemma \ref{lemma:algindsetoutputconnectable} we know that the image of the restricted parameter space
under the path/cycle map $\pi$ has dimension $|E| + |In \cup Out|$, which is equal to the dimension of the image of the full parameter
space under the path/cycle map.  Since the dimension of the image of the coefficient map $c$ is $|E| + |In \cup Out|$,
this must be the same for the restricted model. In particular, if $|L|=|In \cup Out|$, then the model
has $|E| + |In \cup Out|$ parameters, hence it is locally identifiable.
\end{proof}

\begin{ex} [Example \ref{ex:notidpathcycle} revisited] The model $\wcm=(G,\{1\},\{2\},V)$ where $G$ is the graph given by $\{ 1 \rightarrow 2, 3 \rightarrow 2 \}$ is output connectable and has dimension of the image of the coefficient map equal to $|E|+2=4$, thus the model given by $\cm=(G,\{1\},\{2\},\{1,2\})$ is locally identifiable.  
\end{ex}

Combining Theorem \ref{thm:othered} and Theorem \ref{thm:removeleaksoutputconnectable}, we get the following necessary and sufficient conditions:

\begin{cor} \label{cor:necandsuffoutputconnectable} Let $\cm=(G,In,Out,L)$ represent a linear compartmental model with and $L=In \cup Out$ and assume that $G$ is output connectable and $|Out|=1$. Let $\wcm=(G,In,Out,V)$ be the corresponding model with a leak in every compartment.  $\cm$ is locally identifiable if and only if $\wcm$ has dimension of the image of the coeffiicent map as $|E|+|In \cup Out|$.  
\end{cor}

\begin{rmk} Corollary \ref{cor:necandsuffoutputconnectable} shows that this dimension-preserving property when adding or removing leaks from non-input/output compartments also holds in the output connectable case when the dimension of the image of the coefficient map is $|E|+|In \cup Out|$.  
\end{rmk}

\section{Necessary conditions for identifiable models based on model structure} \label{section:neccond}

Outside of checking the conditions in Theorem \ref{thm:almost_isc}, i.e. if a model is an inductively strongly connected model if edges from output to input are added, checking if a model is an identifiable path/cycle model amounts to checking the dimension of the image of the coefficient map, and thus cannot be ascertained by simply examining the graph of the model.  However, it is possible to provide necessary conditions for identifiable models and identifiable path/cycle models based on the graph itself, and thus can be used to rule out identifiability.

\begin{thm} \label{thm:leakcondition} Let $\cm=(G,In,Out,L)$ represent a linear compartmental model with either $G$ strongly input-output connected and $|Out|=1$ or $G$ strongly connected and $|In|=1$.  If $|L|>|In \cup Out|$, then $\cm$ is unidentifiable.  
\end{thm}

\begin{proof} If the number of parameters $|E|+|L| > |E|+|In \cup Out|$, where $|E| + |In \cup Out|$ is the maximal dimension by Lemma \ref{lemma:dimc}, then the model is unidentifiable.  This reduces to $|L| > |In \cup Out|$.  
\end{proof}

We can now make some statements about necessary conditions in the case of the maximal amount of edges.

\begin{thm} Let $\cm=(G,\{i\},\{i\},\{k\})$ represent a linear compartmental model with $G$ strongly connected and $2|V|-2$ edges.  If $\cm$ is locally identifiable (or equivalently, $\wcm=(G,\{i\},\{i\},V)$ is an identifiable cycle model), it must have an exchange.  
\end{thm}

\begin{proof} From Proposition 5.3 of \cite{MeshkatSullivant}, we know that $G$ must have an exchange in order for $\wcm=(G,\{i\},\{i\},V)$ to be an identifiable cycle model.  Thus, if $G$ does not have an exchange, $\wcm$ is not an identifiable cycle model and thus $\cm=(G,\{i\},\{i\},\{k\})$ is not an identifiable model.
\end{proof}

\begin{thm} \label{thm:edgecondition} Let $\cm=(G,\{i\},\{j\},L)$ represent a linear compartmental model with $G$ strongly input-output connected, $\dist(i,j)=1$, and $2|V|-(\dist(i,j) +2)$ edges and $|L|=|In \cup Out|$.  If $\cm$ is locally identifiable (or equivalently, $\wcm=(G,\{i\},\{j\},V)$ is an identifiable path/cycle model), it must have an edge from $i$ to $j$ (i.e. a path of length $\dist(i,j)$). 
\end{thm}

\begin{proof} If there is no path from $i$ to $j$, then the coefficient of the highest-order term on the right hand side of the input-output equation is zero and there would be fewer than $2|V|-\dist(i,j)$ coefficients. But there are $2|V|-(\dist(i,j)+2)$ edges, so if $\cm$ is locally identifiable, then it has expected dimension $|E|+2 = 2|V|-\dist(i,j)$, which is impossible if there are fewer than $2|V|-\dist(i,j)$ coefficients.
\end{proof}

We can also have an analogous necessary condition in the case of fewer than $2|V|-(\dist(i,j)+2)$ edges.

\begin{thm} \label{thm:pathcondition} Let $\cm=(G,\{i\},\{j\},L)$ with $i \neq j$ represent a linear compartmental model with $G$ strongly input-output connected and $2|V|-(k +2)$ edges where $k \geq 1$ and $|L| = |In \cup Out|$.  If $\cm$ is locally identifiable (or equivalently, $\wcm=(G,\{i\},\{j\},V)$ is an identifiable path/cycle model), it must have a path from $i$ to $j$ of length at most $k$.
\end{thm}

\begin{proof} The coefficient of the highest-order term on the right hand side of the input-output equation is a sum of shortest paths from input to output of length $\dist(i,j)$ and this must be nonzero for $\cm$ to have expected dimension.  Since there are $2|V|-\dist(i,j)$ nonzero coefficients and expected dimension is $|E|+2$, this means $k$ is at least $\dist(i,j)$. So there must be a path from $i$ to $j$ of length at most $k$.
\end{proof}

\section{Examples} \label{section:Examples}

We now provide some real world examples that fall into the categories of models considered in this paper.  In particular, we obtain identifiability or unidentifiability results in Example \ref{ex:ex1} and Example \ref{ex:ex2} without any symbolic computation, i.e. purely based on the graph structure alone.  

\begin{ex} \label{ex:ex1} Consider Example 13.6 from \cite{distefano-book} on HIV vaccine development (Part 1).  Three models are considered that fall into the category of path models with leaks from every compartment as in Proposition \ref{prop:path}, shown in Figure 8. The top model corresponds to Experiment 1, the middle model corresponds to Experiment 2, and the bottom model corresponds to Experiment 3. It is clear that the Experiment 1 model is identifiable.  Using Proposition \ref{prop:path}, we can easily obtain that the Experiment 2 model is identifiable.  By Proposition \ref{prop:path}, the Experiment 3 model has expected dimension and is thus unidentifiable with identifiable functions given by the paths and cycles (where the ``self-cycles'' have been expanded out): $k_{23}k'_{12}$, $-k_{03}-k_{23}$, $-k'_{02}-k'_{12}$, $-k''_{01}$.

\end{ex}

\begin{figure}[H] \label{figure:ex1}
\begin{center}

\begin{tikzpicture}[scale=1.2]

 	\draw (4,4) circle (0.3);
    	\node[] at (4, 4) {$1$};


	\draw (4.69,4.69) circle (0.07);	
	 \draw[-] (4.65, 4.65 ) -- (4.22, 4.22);	
	 \draw[->] (3, 4) -- (3.65, 4);	
   	 \node[] at (2.8,4) {in};
	 \draw[->] (4,3.7) -- (4, 3.2);	
   	 \node[] at (4.3, 3.5) {$k_{01}$};
    	

 	\draw (2,2) circle (0.3);
 	\draw (4,2) circle (0.3);
    	\node[] at (2, 2) {2};
    	\node[] at (4, 2) {$1$};
	 \draw[->] (2.35,2) -- (3.65,2);
	\node[] at (3,2.2) {$k_{12}$};
	\draw (4.69,2.69) circle (0.07);	
	 \draw[-] (4.65, 2.65 ) -- (4.22, 2.22);	
	 \draw[->] (1, 2) -- (1.65, 2);	
   	 \node[] at (.8,2) {in};
	 \draw[->] (2,1.7) -- (2, 1.2);	
   	 \node[] at (2.3, 1.5) {$k_{02}$};
	 \draw[->] (4,1.7) -- (4, 1.2);	
   	 \node[] at (4.3, 1.5) {$k'_{01}$};
    	

 	\draw (0,0) circle (0.3);
 	\draw (2,0) circle (0.3);
 	\draw (4,0) circle (0.3);
    	\node[] at (0, 0) {3};
    	\node[] at (2, 0) {2};
    	\node[] at (4, 0) {$1$};
	 \draw[->] (0.35, 0) -- (1.65, 0);
	 \draw[->] (2.35,0) -- (3.65,0);
   	 \node[] at (1, 0.2) {$k_{23}$};
	\node[] at (3,0.2) {$k'_{12}$};
	\draw (4.69,.69) circle (0.07);	
	 \draw[-] (4.65, .65 ) -- (4.22, .22);	
	 \draw[->] (-1, 0) -- (-.35, 0);	
   	 \node[] at (-1.2,0) {in};
	 \draw[->] (0,-.3) -- (0, -.8);	
   	 \node[] at (.3, -.5) {$k_{03}$};
	 \draw[->] (2,-.3) -- (2, -.8);	
   	 \node[] at (2.3, -.5) {$k'_{02}$};
	 \draw[->] (4,-.3) -- (4, -.8);	
   	 \node[] at (4.3, -.5) {$k''_{01}$};
    	
\draw (-2,-1.5) rectangle (5.5, 5);

\end{tikzpicture}
\end{center}
\caption{}
\end{figure}

\begin{ex} \label{ex:ex2} Consider Example 13.16 from \cite{distefano-book} on HIV vaccine development (Part 3).  The models from Example \ref{ex:ex1} are amended by adding on exchanges to the output compartments, shown in Figure 9, and the numbering scheme has changed to agree with \cite{distefano-book}.  The Experiment 1 model is identifiable by Theorem \ref{thm:removeleaks} as it is an identifiable cycle model (it is inductively strongly connected) with a single leak in the input/output compartment.  The Experiment 2 model is almost inductively strongly connected and thus is identifiable by Corollary \ref{cor:almost_isc_id}.  A variation on the Experiment 3 model with leaks from all compartments can be shown to be an identiable path/cycle model by a direct calculation, thus removing the leak from compartment 9 retains the dimension by Theorem \ref{thm:removeleaks}, which means the model in Experiment 3 is unidentifiable.  Alternatively, one can apply Proposition \ref{prop:addvertex} to a variation on the model in Experiment 2 with leaks from every compartment (which has expected dimension) and thus obtain that the variation on the model in Experiment 3 with leaks from every compartment also has expected dimension.  Now removing the leak from compartment 9 retains the dimension by Theorem \ref{thm:removeleaks}, and thus the model in Experiment 3 is unidentifiable.  The identifiable functions are given by the paths and cycles (where the ``self-cycles'' have been expanded out): $k_{53}k_{65}$, $k_{69}k_{96}$, $-k_{03}-k_{53}$, $-k_{05}-k_{65}$, $-k_{06}-k_{96}$, $-k_{69}$, .

\end{ex}

\begin{figure}[H] \label{figure:ex2}
\begin{center}

\begin{tikzpicture}[scale=1.2]

 	\draw (4,4) circle (0.3);
 	\draw (6,4) circle (0.3);

    	\node[] at (4, 4) {$1$};
    	\node[] at (6,4) {$7$};

	 \draw[->] (4.35,4.1) -- (5.65,4.1);
	 \draw[<-] (4.35,3.9) -- (5.65,3.9);

	\node[] at (5,4.3) {$k_{71}$};
	\node[] at (5,3.7) {$k_{17}$};

	\draw (4.69,4.69) circle (0.07);	
	 \draw[-] (4.65, 4.65 ) -- (4.22, 4.22);	
	 \draw[->] (3, 4) -- (3.65, 4);	
   	 \node[] at (2.8,4) {in};
	 \draw[->] (4,3.7) -- (4, 3.2);	
   	 \node[] at (4.3, 3.5) {$k_{01}$};
    	

 	\draw (2,2) circle (0.3);
 	\draw (4,2) circle (0.3);
 	 \draw (6,2) circle (0.3);
    	\node[] at (2, 2) {2};
    	\node[] at (4, 2) {$4$};
    	\node[] at (6,2) {$8$};
	 \draw[->] (2.35,2) -- (3.65,2);
	 \draw[->] (4.35,2.1) -- (5.65,2.1);
	 \draw[<-] (4.35,1.9) -- (5.65,1.9);

	\node[] at (3,2.2) {$k_{42}$};
	\node[] at (5,2.3) {$k_{84}$};
	\node[] at (5,1.7) {$k_{48}$};

	\draw (4.69,2.69) circle (0.07);	
	 \draw[-] (4.65, 2.65 ) -- (4.22, 2.22);	
	 \draw[->] (1, 2) -- (1.65, 2);	
   	 \node[] at (.8,2) {in};
	 \draw[->] (2,1.7) -- (2, 1.2);	
   	 \node[] at (2.3, 1.5) {$k_{02}$};
	 \draw[->] (4,1.7) -- (4, 1.2);	
   	 \node[] at (4.3, 1.5) {$k_{04}$};
    	

 	\draw (0,0) circle (0.3);
 	\draw (2,0) circle (0.3);
 	\draw (4,0) circle (0.3);
 	\draw (6,0) circle (0.3);

    	\node[] at (0, 0) {3};
    	\node[] at (2, 0) {5};
    	\node[] at (4, 0) {$6$};
    	\node[] at (6,0) {$9$};
	 \draw[->] (0.35, 0) -- (1.65, 0);
	 \draw[->] (2.35,0) -- (3.65,0);
	 \draw[->] (4.35,.1) -- (5.65,.1);
	 \draw[<-] (4.35,-.1) -- (5.65,-.1);

   	\node[] at (1, 0.2) {$k_{53}$};
	\node[] at (3,0.2) {$k_{65}$};
	\node[] at (5,.3) {$k_{96}$};
	\node[] at (5,-.3) {$k_{69}$};

	\draw (4.69,.69) circle (0.07);	
	 \draw[-] (4.65, .65 ) -- (4.22, .22);	
	 \draw[->] (-1, 0) -- (-.35, 0);	
   	 \node[] at (-1.2,0) {in};
	 \draw[->] (0,-.3) -- (0, -.8);	
   	 \node[] at (.3, -.5) {$k_{03}$};
	 \draw[->] (2,-.3) -- (2, -.8);	
   	 \node[] at (2.3, -.5) {$k_{05}$};
	 \draw[->] (4,-.3) -- (4, -.8);	
   	 \node[] at (4.3, -.5) {$k_{06}$};
    	

\draw (-2,-1.5) rectangle (7, 5);

\end{tikzpicture}
\end{center}
\caption{}
\end{figure}

\section{Computations} \label{section:computations}

In the table below we outline the number of graphs with $n$ vertices and $m$ edges that have the expected dimension with input in $i$ and output in $j$, assuming leaks from every compartment:

\begin{center}
    \begin{tabular}{ | p{1cm} | p{1.5cm}| p{1.5cm}| p{1.5cm}| p{1.5cm} | p{1.5cm} | p{1.5cm}| p{1.5cm}| p{1.5cm}|}
    \hline
   $(n,m)$ & Total & Strongly Connected &$i=1, \newline	j=1$ &$i=1, \newline j=2,3$ & Strongly input-output connected $i=1, \newline j=2$ & $i=1,\newline j=2$ & Strongly input-output connected $i=1,3, \newline j=2$ & $i=1,3, \newline j=2$  \\ \hline
	  (3,2)  & 15 & NA & NA & NA & 1 & 1 & 3 & 3 \\ \hline
    (3,3)  & 20 & 2 & 2 & 2 & 7 & 4 & 10 & 8 \\ \hline
    (3,4)  & 15 & 9 & 7 & 3 & 11 & NA & 12 & 4 \\ \hline
		(4,3)  & 220 & NA & NA & NA & 2 & 2 & 7 & 7\\ \hline
    (4,4)  & 495 & 6 & 6 & 6 & 37 & 25 & 72 & 59\\ \hline
    (4,5)  & 792 & 84 & 54 & 62 & 193 & 70 & 267 & 167 \\ \hline
    (4,6)  & 924 & 316 & 166 & 118 & 445 & NA & 518 & 184  \\ \hline
		(4,7)  & 792 & 492 & NA & 86 & 565 & NA & 603 & 96 \\ \hline
		(5,4)  & 4845 & NA & NA & NA & 6 & 6 & 24 & 24\\ \hline
    (5,5)  & 15,504 & 24 & 24 & 24 & 222 & 162 & 518 & 432\\ \hline
    (5,6)  & 38,760 & 720 & 576 & 600 & 2470 & 1288 & 4130 & 1110 \\ \hline
    (5,7)  & 77,520 & 6440 & 4052 & 4030 & 13,004 & 3154 & 17,708 & 1552 \\ \hline
    (5,8)  & 125,970 & 26,875 & 9565 & 10,336 & 40,126 & NA & 48,277 & 17,113  \\ \hline
		(5,9) & 167,960 & 65,280 & NA & 15,984 & 82,159 & NA & 91,658 & 20,272 \\ \hline
		(5,10) & 184,756 & 105,566 & NA & 9841 & 120,202 & NA & 128,003 & 10,689 \\ \hline
    \end{tabular}
\end{center}

The number of strongly connected graphs and the number of strongly input-output connected graphs with different input/output configurations is noted.

We then computed the number of models with expected dimension for 4 notable cases: the case of identical single input and single output with a strongly connected graph $G$ (as in \cite{MeshkatSullivant}), the case of single input but multiple outputs with a strongly connected graph $G$, the case of distinct single input and single output with a strongly input-output connected graph $G$, and the case of single output but multiple inputs with a strongly input-output connected graph $G$.  Due to restrictions on the number of edges, not all cases are possible, and those are labeled ``NA''.

\section{Construction of Identifiable Models} \label{section:construction}

In this section, we consider the special case of single input and single output in the same compartment with $G$ strongly connected and $Leak=V$, as in \cite{MeshkatSullivant}, i.e. $\cm=(G,\{i\},\{i\},V)$ with $G$ strongly connected.  Since $Leak$ is assumed to be $V$ and input/output are assumed to be the same vertex, we can just discuss the graph $G$ in what follows.

In \cite{MeshkatSullivant}, Theorem 5.7 gives a way of constructing a new model with expected dimension from a smaller model with expected dimension by adding an incoming and outgoing edge to a chain of cycles (See Definition \ref{defn:chain}):

\begin{thm} [Theorem 5.7 of \cite{MeshkatSullivant}] Let $G'$ be a graph that has the expected dimension with $n-1$ vertices.  Let $G$ be a new graph obtained from $G'$ by adding a new vertex and two edges $k \rightarrow n$ and $n \rightarrow l$ and such that $G$ has a chain of cycles containing both $1$ and $n$.  Then $G$ has the expected dimension.
\end{thm}

In \cite{BaaijensDraisma}, the authors strengthened this result to allow for adding loops of any length, not just length two:

\begin{prop} \label{prop:bd} [Proposition 4.14 of \cite{BaaijensDraisma}] Let $G=(V,E)$ on $n-1$ vertices be a graph with the expected dimension.  Construct $G'$ from $G$ by adding new vertices $n_1,...,n_s$ and edges $k \rightarrow n_1$, $n_s \rightarrow l$, and $n_i \rightarrow n_{i+1}$ for $i=1,...,s-1$ where $k,l \in V$ are vertices of $G$.  Then $G'$ has the expected dimension. 
\end{prop}

We will use Proposition 4.14 from \cite{BaaijensDraisma} combined with Theorem 1 of \cite{MeshkatSullivantEisenberg} to form identifiable models. In other words, we will use Proposition 4.14 to construct identifiable cycle models and then use Theorem 1 to eliminate all but one leak to form an identifiable model. We state Theorem 1 here:

\begin{thm} [Theorem 1 from \cite{MeshkatSullivantEisenberg}] \label{thm:mse}
Let $M$ be an identifiable cycle model.  If the model is changed to have exactly one leak, then the resulting model is locally identifiable.
\end{thm}

\begin{alg} [Construction of identifiable models with $In=Out=\{1\}$ and one leak $|L|=1$] \label{alg:construct}
\ 
\begin{enumerate}
\item Begin with $(G,\{1\},\{1\},\{1\})$ where $V=\{1\}$ and $E= \emptyset$.  
\item Construct $G'$ from $G$ by adding new vertices $n_1,...,n_s$ and edges $1 \rightarrow n_1$, $n_s \rightarrow 1$, and $n_i \rightarrow n_{i+1}$ for $i=1,...,s-1$ where $k,l \in V$ are vertices of $G$ and adding leaks from every new vertex.
\item Repeat Step 2 by starting a some vertex $n_i$ and ending at some vertex $n_j$ for $n_i,n_j \in \{1,n_1,...,n_s\}$ and adding leaks from every new vertex.
\item Continue adding edges, vertices, and leaks as described in Steps 2 and 3.
\item Remove all leaks except one leak.
\end{enumerate}
\end{alg}

\begin{thm} Let $\cm=(G, \{1\}, \{1\}, \{k\})$ be a model constructed from Algorithm \ref{alg:construct}.  The model $\cm$ is identifiable.
\end{thm}

\begin{proof} By Proposition \ref{prop:bd} the model $\cm$ is an identifiable cycle model and by Theorem \ref{thm:mse} the model with only one leak is identifiable.
\end{proof}

\begin{figure}[H]
\begin{center}
	\begin{tikzpicture}[scale=1.2]
 	\draw (0,0) circle (0.3);
	\draw (1.5,-1) circle (0.3);
	\draw (-1.5,-1) circle (0.3);
	\draw (-1.5,-3) circle (0.3);
	\draw (1.5,-3) circle (0.3);	 	
    	\node[] at (0, 0) {1};
    	\node[] at (1.5, -1) {2};
    	\node[] at (-1.5, -1) {$3$};
    	\node[] at (-1.5,-3) {5};
    	\node[] at (1.5,-3) {4};
	 \draw[->] (0.28, -.2) -- (1.2, -.8);
	 \draw[<-] (-.28,-.2) -- (-1.2,-.8);
	 \draw[<-] (-1.15,-1) -- (1.15,-1);
	 \draw[<-] (-1.15,-3) -- (1.15,-3);
	 \draw[<-] (-1.5,-1.35) -- (-1.5,-2.65);
	 \draw[->] (1.5,-1.35) -- (1.5,-2.65);
   	 \node[] at (0, -1.15) {$a_{32}$};
   	 \node[] at (1.25, -2) {$a_{42}$};
   	 \node[] at (-1.25, -2) {$a_{35}$};
   	 \node[] at (0, -2.85) {$a_{54}$};
   	 \node[] at (1, -.3) {$a_{21}$};
   	 \node[] at (-1, -0.3) {$a_{13}$};

     \draw[-] (-.89, .6) -- (-.28, .2);	
	 \draw (-.93,.64) circle (0.07);	
	 \draw[->] (.89, .6) -- (.28, .2);	
   	 \node[] at (1.04,.74) {in};
	 \draw[<-] (0, .95) -- (0, .35);	
   	 \node[] at (.28, .65) {$a_{01}$};
	 \draw[<-] (-2.45, -1) -- (-1.85, -1);	
   	 \node[] at (-2.1, -1.15) {$a_{03}$};
	 \draw[<-] (2.45, -1) -- (1.85, -1);	
   	 \node[] at (2.1, -1.15) {$a_{02}$};
	 \draw[<-] (-2.45, -3) -- (-1.85, -3);	
   	 \node[] at (-2.1, -3.15) {$a_{05}$};
	 \draw[<-] (2.45, -3) -- (1.85, -3);	
   	 \node[] at (2.1, -3.15) {$a_{04}$};
\draw (-3,-4.25) rectangle (3, 1.5);
    	\node[] at (0, -3.85) {$\cm$};
%

\end{tikzpicture}
\end{center}
\caption{The graph corresponding to $\cm$ from Example \ref{ex:construct}.}
\end{figure}
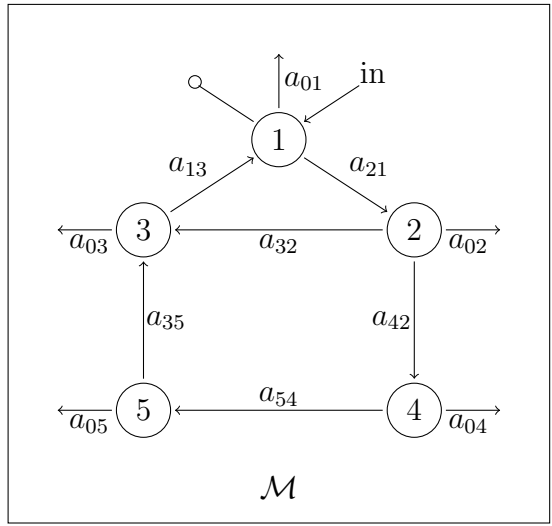

\begin{ex} \label{ex:construct} The model $\cm=(G,\{1\},\{1\},V)$ where $G$ is given by the edges $\{ 1 \rightarrow 2, 2 \rightarrow 3, 3 \rightarrow 1, 2 \rightarrow 4, 4 \rightarrow 5, 5 \rightarrow 3 \}$ is an identifiable cycle model by Proposition \ref{prop:bd}.  Thus we can remove all leaks except one, e.g. $\wcm=(G,\{1\},\{1\},\{5\})$ and the resulting model is identifiable.  We note that this model is \textit{not} inductively strongly connected, thus Proposition \ref{prop:bd} does expand upon the results in \cite{MeshkatSullivantEisenberg} to construct identifiable models.
\end{ex}

Note that the authors in \cite{BaaijensDraisma} only considered \textit{identifiable cycle models}.  We suspect there may be a similar result to Proposition \ref{prop:bd} for the case of identifiable path/cycle models.

\section{Conclusion and Future work} \label{section:conclusion}

In this work, we have defined identifiable path/cycle models and found sufficient conditions for obtaining them.  We have also formed necessary and sufficient conditions for identifiable models with certain graph properties.  Most importantly, we have demonstrated this notion of maximal dimension of the image of the coefficient map in terms of the number of edges and the number of inputs and outputs and have shown that the only identifiable models with certain graph properties are models where the dimension of the image of the coefficient map is preserved to be maximal while adding or subtracting leaks up to a point.  This shows that identifiable models of this form have a special dimension-preserving property even when they are no longer identifiable from added leaks.  We hope that this motivates the investigation of other classes of models that have this same dimension-preserving property.

As mentioned in the introduction, one of the main uses of identifiable functions of parameters is to find identifiable reparametrizations over those functions of parameters.  In \cite{MeshkatSullivant}, necessary and sufficient conditions for an identifiable scaling reparametrization were found and an algorithm to find such a scaling reparametrization was discussed for the special case of single input/single output in the same compartment, $G$ strongly connected, and leaks from every compartment.  We end this work with the following conjecture to generalize these necessary and sufficient conditions for the case of multiple inputs or multiple outputs:

\begin{conj} Let $\cm=(G,In,Out,V)$ represent a linear compartmental model with $G$ strongly input-output connected and $|Out|=1$ or $G$ strongly connected and $|In|=1$.  An identifiable scaling reparametrization exists if and only if $\cm=(G,In,Out,V)$ is an identifiable path/cycle model.
\end{conj}

While the models described in Corollary \ref{cor:necandsuffoutputconnectable} also have the dimension of the image of the coefficient map as $|E|+|In \cup Out|$, we note that these models will not have an identifiable scaling reparametrization since these models do not necessarily have every parameter appearing in a cycle or path from input to output.  However, it may be possible to reparametrize those models another way, e.g. scaling and adding.  We leave these problems of identifiable reparametrizations of models with maximal dimension of the image of their coefficient map for future work.


\section*{Acknowledgments}

This project began at an AIM workshop on ``Identifiability problems in systems biology,'' and the authors thank AIM for providing financial support and an excellent working environment. The authors also thank Seth Sullivant, Anne Shiu, Gleb Pogudin, and two referees for their helpful comments. Cashous Bortner was partially supported by the US National Science Foundation (DMS 1615660).  Nicolette Meshkat was partially supported by the Clare Boothe Luce fellowship from the Henry Luce Foundation and the US National Science Foundation (DMS 1853525).  

\bibliographystyle{plain}
\bibliography{AIM2}

\end{document}